\documentclass[english]{article}
\usepackage[T1]{fontenc}
\usepackage[latin9]{inputenc}
\usepackage{amsthm}
\usepackage{amsmath}
\usepackage{amssymb}
\usepackage{esint}

\makeatletter
\numberwithin{equation}{section}
\numberwithin{figure}{section}
  \theoremstyle{definition}
  \newtheorem{defn}{\protect\definitionname}[section]
  \theoremstyle{plain}
  \newtheorem{thm}{\protect\theoremname}[section]
  \theoremstyle{definition}
  \newtheorem{example}{\protect\examplename}[section]
  \theoremstyle{plain}
  \newtheorem{prop}{\protect\propositionname}[section]
  \theoremstyle{remark}
  \newtheorem{rem}{\protect\remarkname}[section]
  \theoremstyle{plain}
  \newtheorem{cor}{\protect\corollaryname}[section]
  \theoremstyle{plain}
  \newtheorem{lem}{\protect\lemmaname}[section]

\DeclareSymbolFont{yhlargesymbols}{OMX}{yhex}{m}{n}
\DeclareMathAccent{\wideparen}{\mathord}{yhlargesymbols}{"F3}

\usepackage{graphicx}

\makeatother

\usepackage{babel}
  \providecommand{\definitionname}{Definition}
  \providecommand{\examplename}{Example}
  \providecommand{\lemmaname}{Lemma}
  \providecommand{\propositionname}{Proposition}
  \providecommand{\remarkname}{Remark}
\providecommand{\corollaryname}{Corollary}
\providecommand{\theoremname}{Theorem}

\begin{document}

\title{Reconstruction for the Signature of a Rough Path}

\author{Xi Geng%
\thanks{Department of Mathematical Sciences, Carnegie Mellon University, Pittsburgh, PA 15213, United States. Email: xig@andrew.cmu.edu. %
}}
\date{}

\maketitle
\begin{abstract}
Recently it was proved that the group of rough paths modulo tree-like
equivalence is isomorphic to the corresponding signature group through
the signature map $S$ (a generalized notion of taking iterated path
integrals). However, the proof of this uniqueness result does not
contain any information on how to ``see'' the trajectory of a (tree-reduced)
rough path from its signature, and a constructive understanding on
the uniqueness result (in particular on the inverse of $S$) has become
an interesting and important question. The aim of the present paper
is to reconstruct a rough path from its signature in an explicit and
universal way.
\end{abstract}

\smallskip
\noindent$\mathrm{MSC\ 2010:}$ 34F05(primary), 60G17, 60H10(secondary)

\smallskip
\noindent $\mathrm{Keywords:}$ Rough paths; Uniqueness of signature; Reconstruction

\section{Introduction}

Motivated from the study of homotopy theory and cohomology of loop
spaces, in 1954 K.T. Chen \cite{Chen1954} introduced the powerful
tool of iterated path integrals. In particular, he observed that a
continuous path with bounded variation can be represented by a fully
non-commutative power series through an exponential homomorphism.
In terms of tensor products, this is equivalent to saying that the
integration map $S$, which sends a path $x$ in $\mathbb{R}^{d}$
to the formal tensor series
\[
1+\sum_{n=1}^{\infty}\int_{0<t_{1}<\cdots<t_{n}<1}dx_{t_{1}}\otimes\cdots\otimes dx_{t_{n}},
\]
is a homomorphism from the semigroup of paths under concatenation
to the algebra of formal tensor series under tensor product. Moreover,
in 1957 Chen \cite{Chen1957} discovered an important algebraic property
of such representation which asserts that the formal logarithm of
$S$ is always a formal Lie series. Equivalently, $S(x)$ satisfies
the shuffle product formula for every path $x$.

Inspired by the algebraic structure of iterated path integrals, in
a seminal work \cite{Lyons1998} in 1998, T. Lyons developed a theory
of path integration and differential equations driven by rough signals,
which is now known as \textit{rough path theory}. The key novel point in his
theory is that if a path is irregular, a collection of ``generalized
iterated integrals'' up to a certain degree should be pre-specified
in order to define integration against such a path. More precisely, a generic path should be a path taking values in the free nilpotent
group up to a certain degree, which is related to the roughness of the
underlying path. Identifying the right topology for paths
with certain roughness to ensure continuity properties for differential equations is a fundamental contribution in the analytic
aspect of rough path theory.

The development of rough path theory leads to tremendous applications
in probability theory, due to the important fact that most of interesting
continuous stochastic processes can be regarded as rough paths in
a canonical way. It follows that a pathwise theory of stochastic differential
equations is a consequence of rough path theory. This provides a new
perspective in solving a lot of probabilistic problems, for instance
regularity of hypoelliptic Gaussian SDEs \cite{CF2010}, large deviation
principles \cite{LQZ2002}, analysis on path and loop spaces \cite{Aida2011}
and etc. Rough path theory also provides a one dimensional prototype
of M. Hairer's Fields medal work on the theory of regularity structures
for stochastic partial differential equations.

An important result in rough path theory, known as Lyons' extension
theorem, asserts that a rough path has a unique lifting to the full
tensor algebra which has the same regularity. In particular, a rough
path can also be represented by a formal tensor series of ``generalized
iterated integrals'' as in the bounded variation case. In the rough
path literature, such representation is usually known as the signature
of a rough path.

A fundamental question in rough path theory is to understand in what
sense is such representation faithful. In 1958, Chen \cite{Chen1958}
gave an answer to this question for the class of piecewise regular
and irreducible paths. After five decades, in 2010 H. Bambly and T.
Lyons \cite{HL2010} solved this problem for the class of continuous
paths with bounded variation. More precisely, they proved that a continuous
path of bounded variation is determined by its signature up to tree-like
equivalence. This result was recently extended to the rough path setting by
H. Boedihardjo, X. Geng, T. Lyons and D. Yang \cite{BGLY2014} in
2016. 

From a theoretical point of view, the uniqueness result of \cite{BGLY2014}
is important as it builds an isomorphism between the rough path space
modulo tree-like equivalence and the corresponding signature group.
In particular, it reveals a link between the geometric property of
rough paths and the algebraic property of signatures. A natural reason
of looking at the signature is that the algebraic structure is very
explicit and simple: polynomial functionals on the signature group
are always linear functionals. Therefore, regular functions on
signature are  relatively easy to study. On the practical side,
this point plays a key role in the signature approach for time series
analysis and financial data analysis (c.f. \cite{LLN2013}, \cite{LNO2014}
for instance).

On the other hand, despite of the fact that the signature map $S$
descends to an isomorphism between tree-reduced rough paths (a canonical
representative in each equivalence class) and their signatures, the
proof of the uniqueness result in \cite{BGLY2014} does not contain
any information about how one can ``see'' the trajectory of a tree-reduced
rough path from its signature. The question about reconstructing of a path
from its signature is interesting and important as it might shed light on understanding the local geometry of a path from global information (the signature). Moreover, a nice reconstruction method might be useful to understand functions on the rough path space by pulling them back to the signature group through
the inverse map $S^{-1},$ and as mentioned before, the resulting
functions on signature are generally easier to study. The reconstruction problem was first studied by T. Lyons and W. Xu
\cite{LX2015A}, \cite{LX2015B} for the class of piecewise linear
paths and the class of continuously differentiable paths in which
the modulus of continuity for the derivatives is known. 

The aim of the present paper is to study the reconstruction problem
in the general rough path setting. To be precise, we aim at reconstructing
a tree-reduced rough path from its signature in an explicit and universal
way, in the sense that it relies only on the Euclidean structure where
our underlying paths should live, and with the knowledge of any given
signature, our reconstruction produces the underlying tree-reduced
rough path. The main idea of our reconstruction is motivated from
the understanding on Y. Le Jan and Z. Qian's work \cite{LQ2013} for
Brownian motion into a completely deterministic setting (c.f. Section 3 for a
more detailed explanation). We hope that our work might give us a
constructive understanding on the uniqueness result in \cite{BGLY2014},
and in particular on the inverse of the signature map.

The present paper is organized in the following way. In Section 2,
we recall the basic notions on rough paths and the uniqueness result
in \cite{BGLY2014}. In Section 3, we explain the main idea of our
reconstruction, and provide the precise mathematical setting of our
problem. In Section 4, we present several preliminary results which
are essential for our study. In Section 5, we consider the case when
our underlying tree-reduced rough paths are non-self-intersecting.
In Section 6, we deal with the general case. In Section 7, we give
a few concluding remarks on the present work.

\section{Generalities on the Uniqueness Result for the Signature of a Rough
Path}

In this section, we present the basic notions on the signature of
rough paths and recapture the main result of \cite{BGLY2014}. We
will see that the general idea of understanding non-self-intersecting signature paths in \cite{BGLY2014} plays an important role in the present work.

Let $T((\mathbb{R}^{d}))$ be the infinite dimensional tensor algebra
over $\mathbb{R}^{d}$ consisting of formal series of tensors in each
degree. For $N\in\mathbb{N},$ let 
\[
T^{(N)}(\mathbb{R}^{d})=\bigoplus_{k=0}^{N}(\mathbb{R}^{d})^{\otimes k}
\]
be the truncated tensor algebra up to degree $N$. Here we identify
$(\mathbb{R}^{d})^{\otimes k}\cong\mathbb{R}^{d^{k}}$ with basis
$\{e_{i_{1}}\otimes\cdots\otimes e_{i_{k}}:\ 1\leqslant i_{1},\cdots,i_{k}\leqslant d\}$,
where $\{e_{1},\cdots,e_{d}\}$ is the canonical basis of $\mathbb{R}^{d}.$ 
\begin{defn}
\label{def: p-rough path}A \textit{multiplicative functional} of
degree $N\in\mathbb{N}$ is a continuous map $\mathbf{X}_{\cdot,\cdot}=(1,X_{\cdot,\cdot}^{1},\cdots,X_{\cdot,\cdot}^{N})$
from the standard $2$-simplex $\Delta=\{(s,t):\ 0\leqslant s\leqslant t\leqslant1\}$
to $T^{(N)}(\mathbb{R}^{d})$ satisfying the following so-called \textit{Chen's
identity}:
\[
\mathbf{X}_{s,u}=\mathbf{X}_{s,t}\otimes\mathbf{X}_{t,u},\ \forall0\leqslant s\leqslant t\leqslant u\leqslant1.
\]
Let $\mathbf{X},\mathbf{Y}$ be two multiplicative functionals of
degree $N.$ Define 
\[
d_{p}(\mathbf{X},\mathbf{Y})=\max_{1\leqslant i\leqslant N}\sup_{\mathcal{P}_{[0,1]}}\left(\sum_{l}\left|X_{t_{l-1},t_{l}}^{i}-Y_{t_{l-1},t_{l}}^{i}\right|^{\frac{p}{i}}\right)^{\frac{i}{p}},
\]
where $\mathcal{P}_{[0,1]}$ denotes all finite partitions of $[0,1].$
$d_{p}$ is called the $p$-\textit{variation metric}, and we say
$\mathbf{X}$ has finite (total) $p$-variation if $d_{p}(\mathbf{X},\mathbf{1})<\infty$
where $\mathbf{1}:=(1,0,\cdots,0).$ A multiplicative functional $\mathbf{X}$
of degree $\lfloor p\rfloor$ with finite  $p$-variation is called
a $p$-\textit{rough path}. The space of $p$-rough paths over $\mathbb{R}^{d}$
is denoted by $\Omega_{p}(\mathbb{R}^{d}).$
\end{defn}

A fundamental analytic property of $p$-rough paths is the following
extension theorem proved by Lyons \cite{Lyons1998}. It asserts that
some analytic version of ``iterated path integrals'' against a $p$-rough
path can be uniquely defined and has a nice factorial decay property. 
\begin{thm}
\label{thm: Lyons' extension}(Lyons' extension theorem) Let $\mathbf{X}$
be a $p$-rough path. Then for any $i\geqslant\lfloor p\rfloor+1,$
there exists a unique continuous map $X^{i}:\ \Delta\rightarrow\left(\mathbb{R}^{d}\right)^{\otimes i}$
such that 
\[
S(\mathbf{X}):=\left(1,X^{1},\cdots,X^{\lfloor p\rfloor},\cdots X^{i},\cdots\right)
\]
is a multiplicative functional in \textup{$T\left(\left(\mathbb{R}^{d}\right)\right)$
with finite  $p$-variation when restricted up to each degree. Moreover,
there exists a positive constant $C_{p}$ depending only on $p$,
such that 
\begin{equation}
\left|X_{s,t}^{i}\right|\leqslant\frac{C_{p}\omega(s,t)^{\frac{i}{p}}}{\left(\frac{i}{p}\right)!},\ \forall i\geqslant1\ \mathrm{and}\ \forall(s,t)\in\Delta,\label{eq: Lyons' extension}
\end{equation}
where $\omega$ is the control function defined by 
\[
\omega(s,t)=\sum_{i=1}^{\lfloor p\rfloor}\sup_{\mathcal{P}_{[s,t]}}\sum_{l}\left|X_{t_{l-1},t_{l}}^{i}\right|^{\frac{p}{i}},\ (s,t)\in\Delta.
\]
}
\end{thm}

\begin{defn}
The tensor element $S(\mathbf{X})_{0,1}\in T((\mathbb{R}^{d}))$ defined
by Theorem \ref{thm: Lyons' extension} is called the \textit{signature}
of $\mathbf{X}.$
\end{defn}

In the case of $p=1,$ a $p$-rough path $\mathbf{X}$ is just the
collection of increments of the continuous path $x_{\cdot}:=\mathbf{X}_{0,\cdot}$
in $\mathbb{R}^{d}$ with bounded  variation, and Lyons' extension
of $\mathbf{X}$ reduces to the classical iterated path integrals
along $x$. 

Now let $G(\mathbb{R}^{d})$ be the group
defined by the exponential of formal Lie series in $T((\mathbb{R}^{d})).$
For $N\in\mathbb{N},$ let $G^{N}(\mathbb{R}^{d})$ be the truncation
of $G(\mathbb{R}^{d})$ up to degree $N$, which consists of the exponential
of Lie polynomials up to degree $N$ in $T^{(N)}(\mathbb{R}^{d}).$
$G^{N}(\mathbb{R}^{d})$ is usually known as the \textit{free nilpotent group}
of degree $N.$

The following result, known as the \textit{shuffle product formula}, reveals
a fundamental algebraic property for the signature of continuous paths with bounded
 variation (c.f. Reutenauer \cite{Reutenauer1993} and also Chen \cite{Chen1957}).
This property lies as a crucial base to expect that
a path with bounded variation is uniquely determined by its signature up to tree-like equivalence and also to its extension to the rough path setting.
\begin{thm}
\label{thm: shuffle product formula}(Shuffle product formula) A tensor
element $a=(1,a^{1},a^{2},\cdots)\in T((\mathbb{R}^{d}))$ belongs
to $G(\mathbb{R}^{d})$ if and only if 
\begin{equation}
a^{m}\otimes a^{n}=\sum_{\sigma\in\mathcal{S}(m,n)}\mathcal{P}^{\sigma}(a^{m+n}),\ \forall m,n\geqslant1,\label{eq: shuffle product formula}
\end{equation}
where $\mathcal{S}(m,n)$ denotes the set of $(m,n)$-shuffles in
the permutation group of order $m+n$, and $\mathcal{P}^{\sigma}:\ V^{\otimes(m+n)}\rightarrow V^{\otimes(m+n)}$
is the permutation operator given by 
\[
\mathcal{P}^{\sigma}(v_{1}\otimes\cdots\otimes v_{m+n})=v_{\sigma(1)}\otimes\cdots\otimes v_{\sigma(m+n)}.
\]

In particular, for any continuous path $x$ in $\mathbb{R}^{d}$ with
bounded variation, the signature of $x$ satisfies (\ref{eq: shuffle product formula})
and hence belongs to $G(\mathbb{R}^{d}).$
\end{thm}

Following the general setting of \cite{BGLY2014}, in the present
paper we work with a class of rough paths called weakly geometric
rough paths. It is defined on the previous algebraic structure and is the
fundamental class of paths that the theory of rough integration and
differential equations is based on.
\begin{defn}
\label{def: weakly geometric rough paths}A \textit{weakly geometric
$p$-rough path} is a $p$-rough path taking values in $G^{\lfloor p\rfloor}(\mathbb{R}^{d}).$
The space of weakly geometric $p$-rough paths over $\mathbb{R}^{d}$
is denoted by $WG\Omega_{p}(\mathbb{R}^{d}).$
\end{defn}

It can be shown (c.f. \cite{FV2010}, and also \cite{BGLY2014}) that
the signature of a weakly geometric rough path is an element in $G.$

\begin{rem}
There is an equivalent intrinsic definition of weakly geometric rough
paths in terms of the Carnot\textendash Carathéodory metric $d_{CC}$
on $G^{\lfloor p\rfloor}(\mathbb{R}^{d}).$ To be precise, a weakly geometric
$p$-rough path is a continuous path $\mathbf{X}:\ [0,1]\rightarrow G^{\lfloor p\rfloor}(\mathbb{R}^{d})$
starting at the unit such that $\mathbf{X}$ has finite  $p$-variation
with respect to the metric $d_{CC}$. This definition is equivalent
to Definition \ref{def: weakly geometric rough paths} through $\mathbf{X}_{s,t}=X_{s}^{-1}\otimes\mathbf{X}_{t}$
and $\mathbf{X}_{t}=\mathbf{X}_{0,t}$. We refer the reader to \cite{FV2010}
for a detailed discussion along this approach.
\end{rem}

From now on, unless otherwise stated we always regard a weakly geometric $p$-rough path as
an actual path in $G^{\lfloor p\rfloor}(\mathbb{R}^{d})$ instead
of a multiplicative functional defined on the simplex $\Delta.$

The main result of \cite{BGLY2014} states that a weakly geometric
$p$-rough path is uniquely determined by its signature up to tree-like
equivalence. More precisely, it was shown that:
\begin{thm}
\label{thm: uniqueness result}A weakly geometric $p$-rough path
$\mathbf{X}$ has trivial signature if and only if it is tree-like,
in the sense that there exists some real tree $\tau$ together with
a continuous loop $\alpha:\ [0,1]\rightarrow\tau$ and some continuous
map $\psi:\ \tau\rightarrow WG\Omega_{p}(\mathbb{R}^{d})$ such that
$\mathbf{X}=\psi\circ\alpha.$ 
\end{thm}

In particular, let 
\[
\mathfrak{S}_{p}=\{g=S(\mathbf{X})_{0,1}:\ \mathbf{X}\in WG\Omega_{p}(\mathbb{R}^{d})\}
\]
be the signature group for weakly geometric $p$-rough paths. It was
proved that for each $g\in\mathfrak{S}_{p},$ there exists a unique
$\mathbf{X}^{g}\in WG\Omega_{p}(\mathbb{R}^{d})$ up to reparametrization,
such that $S(\mathbf{X}^{g})_{0,1}=g$ and $S(\mathbf{X}^{g})_{0,\cdot}$
is a non-self-intersecting (or simple) path in $T((\mathbb{R}^{d}))$.
$\mathbf{X}^{g}$ is known as the \textit{tree-reduced} path associated with
signature $g.$ It follows that $\mathfrak{S}_{p}$ can be equipped
with a real tree metric in a canonical way, and a weakly geometric
$p$-rough path with trivial signature factors through this real tree
with $\psi$ being the projection map.

Therefore, modulo tree-like paths the signature homomorphism 
\begin{eqnarray*}
S_{p}:\ WG\Omega_{p}(\mathbb{R}^{d}) & \rightarrow & T((\mathbb{R}^{d})),\\
\mathbf{X} & \mapsto & S(\mathbf{X})_{0,1},
\end{eqnarray*}
descents to a group isomorphism
\[
\widetilde{S}_{p}:\ WG\Omega_{p}(\mathbb{R}^{d})/_{\mathrm{tree-like}}\stackrel{\cong}{\rightarrow}\mathfrak{S}_{p}.
\]
Moreover, each equivalence class contains a unique tree-reduced path
$\mathbf{X}$ (up to reparametrization) and this isomorphism gives rise to a one-to-one
correspondence between tree-reduced paths (up to reparametrization)
and signatures.

However, the proof of Theorem \ref{thm: uniqueness result} contains
no information about how a tree-reduced path can be reconstructed
from its signature. The development of such a reconstruction in an
explicit and universal way is the main focus of the present paper.

Before studying the reconstruction problem, let us mention the following
interesting fact. In the case of $p=1,$ it is not hard to see that
(c.f. \cite{HL2010}) a path is tree-reduced if and only if it is a
reparametrization of the unique minimizer of $1$-variation (i.e.
the length) among paths parametrized by arc length with the same signature.
However, this is not true in the case of $p>1$. A tree-reduced path
certainly minimizes the $p$-variation, but a $p$-variation minimizer
might \textit{not} be tree-reduced no matter how it is parametrized.
We conclude this section by providing such an example.
\begin{example}
Consider the two dimensional case and $1<p<2.$ Let $\wideparen{AB}$
be an arc of the unit circle centered at $O\in\mathbb{R}^{2}$ with
central angle $\theta_{0}$, and let $C$ be the midpoint of $\wideparen{AB}$.
Let $D$ be a point on the extension of the radius vector $\overrightarrow{OC}$
and let $|CD|=\varepsilon.$ Consider the paths $x,y:\ [0,1]\rightarrow\mathbb{R}^{2}$
defined by the trajectories 
\[
x=\wideparen{AC}\sqcup\overrightarrow{CD}\sqcup\overrightarrow{DC}\sqcup\wideparen{CB},\ y=\wideparen{AB},
\]
respectively, where ``$\sqcup$'' means concatenation. It is easy
to see that $x,y$ have the same signature $g$ and $y$ is a tree-reduced
path. Apparently $x$ and $y$ do not differ by just a reparametrization.
Now we show that $\|x\|_{p-\mathrm{var}}=\|y\|_{p-\mathrm{var}}=\left|\overrightarrow{AB}\right|$
provided $\theta_{0}$ and $\varepsilon$ are small enough. 

In fact, let $E\in\wideparen{AB}$ and denote the central angle $\angle EOB$
by $\theta.$ Consider the function
\[
f(\theta)=\left|\overrightarrow{AE}\right|^{p}+\left|\overrightarrow{EB}\right|^{p},\ \theta\in[0,\theta_{0}],
\]
which can be written as 
\[
f(\theta)=2^{p}\left(\sin^{p}\frac{\theta}{2}+\sin^{p}\frac{\theta_{0}-\theta}{2}\right)
\]
according to Euclidean geometry. Computing the second derivative of
$f$, we obtain that
\[
f''(\theta)=\frac{p}{2^{2-p}}\left((p-1)\left(\frac{\cos^{2}\frac{\theta}{2}}{\sin^{2-p}\frac{\theta}{2}}+\frac{\cos^{2}\frac{\theta_{0}-\theta}{2}}{\sin^{2-p}\frac{\theta_{0}-\theta}{2}}\right)-\left(\sin^{p}\frac{\theta}{2}+\sin^{p}\frac{\theta_{0}-\theta}{2}\right)\right).
\]
Since $1<p<2,$ we know that when $\theta_{0}$ is small, $f''(\theta)$
is uniformly positive and hence $f$ is convex on $[0,\theta].$ Also
note that $f(0)=f(\theta_{0})=\left|\overrightarrow{AB}\right|^{p}$.
Therefore, for $\theta_{0}$ small enough $f$ obtains its maximum
on the end points and we have 
\[
\left|\overrightarrow{AE}\right|^{p}+\left|\overrightarrow{EB}\right|^{p}\leqslant\left|\overrightarrow{AB}\right|^{p},\ \forall E\in\wideparen{AB}.
\]
Now we fix such $\theta_{0}$. This already implies that $\|y\|_{p}=\left|\overrightarrow{AB}\right|$.
Moreover, by considering the symmetry of $f(\theta)$ it is easy to
see that $f$ obtains its minimum at $\theta=\theta_{0}/2.$ Set 
\[
\lambda=\left|\overrightarrow{AB}\right|^{p}-\left|\overrightarrow{AC}\right|^{p}-\left|\overrightarrow{CB}\right|^{p}>0.
\]

It remains to show that when $\varepsilon$ is small enough, $\|x\|_{p}=\left|\overrightarrow{AB}\right|$.
To this end, let 
\[
\mathcal{P}:\ 0=t_{0}<t_{1}<\cdots<t_{n}=1
\]
be a finite partition of $[0,1]$, and let $t_{k},t_{l}$ be the first
and last partition points at which $x$ is in $\overline{CD}$ respectively.
If such points do not exist, then obviously we have 
\[
\sum_{i=1}^{n}\left|x_{t_{i}}-x_{t_{i-1}}\right|^{p}\leqslant\left|\overrightarrow{AB}\right|^{p}.
\]
Otherwise, we have 
\begin{eqnarray*}
\sum_{i=1}^{n}\left|x_{t_{i}}-x_{t_{i-1}}\right|^{p} & = & \sum_{i=1}^{k-1}\left|x_{t_{i}}-x_{t_{i-1}}\right|^{p}+\left|x_{t_{k}}-x_{t_{k-1}}\right|^{p} +\sum_{i=k+1}^{l}\left|x_{t_{i}}-x_{t_{i-1}}\right|^{p}\\
 &  & +\left|x_{t_{l+1}}-x_{t_{l}}\right|^{p}+\sum_{i=l+2}^{n}\left|x_{t_{i}}-x_{t_{i-1}}\right|^{p}\\
 & \leqslant & \left|x_{t_{k-1}}-A\right|^{p}+\left|B-x_{t_{l+1}}\right|^{p}+\left|x_{t_{k}}-x_{t_{k-1}}\right|^{p}\\
 &  & +\left|x_{t_{l+1}}-x_{t_{l}}\right|^{p}+2\varepsilon^{p},
\end{eqnarray*}
where we have used the previous discussion and the fact that $\overrightarrow{CD}$
is a geodesic. It follows that
\begin{align*}
 & \sum_{i=1}^{n}\left|x_{t_{i}}-x_{t_{i-1}}\right|^{p}\\
\leqslant & \left|\overrightarrow{AC}\right|^{p}+\left|\overrightarrow{CB}\right|^{p}+\left(\left|x_{t_{k}}-x_{t_{k-1}}\right|^{p}-\left|C-x_{t_{k-1}}\right|^{p}\right)\\
 & +\left(\left|x_{t_{l+1}}-x_{t_{l}}\right|^{p}-\left|x_{t_{l+1}}-C\right|^{p}\right)+2\varepsilon^{p}\\
= & \left|\overrightarrow{AB}\right|^{p}-\left(\left|\overrightarrow{AB}\right|^{p}-\left|\overrightarrow{AC}\right|^{p}-\left|\overrightarrow{CB}\right|^{p}-\left(\left|x_{t_{k}}-x_{t_{k-1}}\right|^{p}-\left|C-x_{t_{k-1}}\right|^{p}\right)\right.\\
 & \left.-\left(\left|x_{t_{l+1}}-x_{t_{l}}\right|^{p}-\left|x_{t_{l+1}}-C\right|^{p}\right)-2\varepsilon^{p}\right).
\end{align*}
On the other hand, we have 
\begin{eqnarray*}
\left|x_{t_{k}}-x_{t_{k-1}}\right|^{p}-\left|C-x_{t_{k-1}}\right|^{p} & \leqslant & \left(\left|C-x_{t_{k-1}}\right|+\varepsilon\right)^{p}-\left|C-x_{t_{k-1}}\right|^{p}\\
 & \leqslant & \max\left\{ \left(\sqrt{\varepsilon}+\varepsilon\right)^{p},\theta_{0}^{p}\left(\left(1+\sqrt{\varepsilon}\right)^{p}-1\right)\right\} \\
 & =: & \mu(\varepsilon).
\end{eqnarray*}
The same inequality holds for $\left|x_{t_{l+1}}-x_{t_{l}}\right|^{p}-\left|x_{t_{l+1}}-C\right|^{p}.$
Therefore, we have
\begin{eqnarray*}
\sum_{i=1}^{n}\left|x_{t_{i}}-x_{t_{i-1}}\right|^{p} & \leqslant & \left|\overrightarrow{AB}\right|^{p}-\left(\lambda-2\mu(\varepsilon)-2\varepsilon^{p}\right)\leqslant\left|\overrightarrow{AB}\right|^{p},
\end{eqnarray*}
provided $\varepsilon$ is small enough so that 
\[
2\mu(\varepsilon)+2\varepsilon^{p}<\lambda.
\]
Now by taking supremum over all finite partitions of $[0,1],$ we
conclude that 
\[
\|x\|_{p}\leqslant\left|\overrightarrow{AB}\right|=\|y\|_{p}\leqslant\|x\|_{p}.
\]

Therefore, $x$ is also a $p$-variation minimizer with signature
$g.$
\end{example}

\section{Main Idea of the Reconstruction}

In this section, we explain the basic idea of our reconstruction and
present the precise mathematical setting of our problem.

The underlying strategy of our reconstruction is motivated from the probabilistic
work of Le Jan and Qian \cite{LQ2013} for Brownian motion, which was further developed
by X. Geng Z. Qian \cite{GQ2015} for hypoelliptic diffusions and
by H. Boedihardjo and X. Geng \cite{BG2015} in the non-Markov setting.
We first summarize the original idea of Le Jan and Qian for the almost-sure
reconstruction of Brownian sample paths.

(1) By knowing the (Stratonovich) signature of Brownian sample paths,
the shuffle product formula allows us to construct iterated path integrals
along any finite sequence of regular one forms (which they called
extended signatures).

(2) Given a disjoint family of nice compact
domains $\{K_{n}\}$ in space with a well-chosen one form supported
in each $K_{n},$ by evaluating extended signatures associated with
this family of one forms, with probability one we can determine the
ordered sequence of domains $K_{n}$ visited by a Brownian path.

(3) Construct a polygonal path associated with this ordered sequence
of domains $K_{n}$ visited by the underlying Brownian path. As we refine the geometric scheme, it is reasonable to expect that the polygonal approximation converges to the original Brownian path in some sense.

In the probabilistic setting, the way of choosing the ``testing''
one forms and of refining the geometric scheme to obtain convergence
depends on the a priori knowledge on the law of the underlying process, in particular on its certain non-degeneracy properties. In this respect, it gives
rise to the main obstruction of developing such idea in the deterministic
setting for arbitrary weakly geometric rough paths as we need
to treat \textit{every} path equally. We explain this point in more details. 

For simplicity, let us consider the two dimensional case in which
our underlying paths are non-self-intersecting and have bounded variation. Assume
that the plane is decomposed into disjoint squares of size order $\varepsilon$
and narrow tunnels of width order $\delta$ ($\delta<<\varepsilon$). By using extended signatures
we can determine the ordered sequence $\mathbf{m}^{\varepsilon,\delta}$
of squares visited by a path (this relies on the non-self-intersecting assumption
on our path in a crucial way), and we can construct a polygonal path
associated with this sequence according to (3) as discussed before. 

Now if we refine our geometric scheme by letting $(\varepsilon,\delta)\rightarrow0$
independently, we \textit{cannot} expect that the polygonal
approximation would converge to the original path in any sense although it
should be true generically. In general, for any given refinement of geometric schemes, one could always construct a path such that the convergence fails for this path. Heuristically, such a path should have a long excursion in the complement of each geometric scheme (i.e. the tunnels in this case) along the refinement.

The following is a simple example illustrating this situation.

\begin{example}
Figure \ref{fig: example that convergence fails} gives a simple example that the convergence of the polygonal approximation could fail in general. Here $x$ is a piecewise linear path starting from origin with two edges, going first from $(0,0)$ to $(1,0)$ and then from $(1,0)$ to $(1,1)$. In the $\varepsilon$-scale, every square is centered at $\varepsilon\boldsymbol{n}$ for some integer point $\boldsymbol{n}$ in the plane, leaving gaps between squares to the order of another independent parameter $\delta$. Now choose a subsequence $\varepsilon_n=\frac{2}{2n+1}$ and arbitrary $\delta_n$. It is then easy to see that along this subsequence $\varepsilon_n$, the polygonal approximation $x^n$ is given by the linear path joining the origin and the point $n\varepsilon_n$. In this situation we cannot expect that $x^n$ converges to the original path $x$ in any reasonable sense. 
\end{example}

\begin{figure} 
\begin{center} 
\includegraphics[scale=0.4]{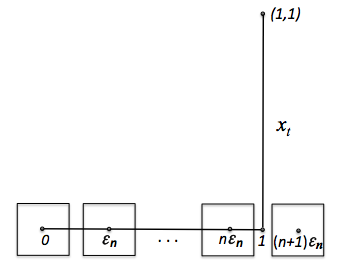}
\protect\caption{An example illustrating that the convergence of polygonal approximation can fail as $(\varepsilon,\delta)\rightarrow0.$}\label{fig: example that convergence fails}
\end{center}
\end{figure}

In a more general situation, suppose that $\{\mathcal{C}_n\}$ is any given refinement of geometric schemes in $\mathbb{R}^d$ (i.e. the size of domains goes to zero uniformly as $n\rightarrow\infty$), and let $0<r<R$. Then one can prove that there exists a subsequence ${\mathcal{C}_{l_n}}$ and a continuous path $x:\ [0,1]\rightarrow\mathbb{R}^d$, such that $|x_1|\geqslant R$ but the polygonal approximation $x^n$ of $x$ with respect to $\mathcal{C}_{l_n}$ is contained in $B(0,r)$ for all sufficiently large $n$. Therefore, one could not expect that $x^n$ converges to $x$ in any reasonable sense. To get an idea of how this works, we can first pick some geometric scheme $\mathcal{C}_{l_1}$ and construct a path $x^1$ with $|x^1_1|\geqslant R$ but which does not enter domains in $\mathcal{C}_{l_1}$ which are outside $B(0,r)$. Then we can pick some next geometric scheme ${\mathcal{C}_{l_2}}$ and modify $x^1$ to another path $x^2$ with $|x_1^2|\geqslant R$ but which does not enter domains in $\mathcal{C}_{l_1}$ \textit{and} $\mathcal{C}_{l_2}$ which are outside $B(0,r)$. By doing this inductively and keeping track of the uniform error created in each step, we can obtain a sequence of paths $x^n$ which converges uniformly to some limit path $x$ with $|x_1|\geqslant R$, and $x^n$ does not enter domains in $\mathcal{C}_{l_1},\cdots,\mathcal{C}_{l_n}$ which are outside $B(0,r)$. It is then natural to expect that the limit path $x$ does not enter domains in $\mathcal{C}_{l_n}$ which are outside $B(0,r)$ for every $n$. Therefore, $x$ will be a desired counter-example. The detailed proof of this claim is quite technical and is hence omitted.

The
main idea of overcoming  this issue is to let the signature $g$ determine
the width $\delta$ of the tunnels for each $\varepsilon$ according
to some stable quantity, so that the underlying path with signature
$g$ is not able to travel through any long narrow tunnels in the
resulting geometric scheme. Therefore the approximating polygonal
path must converge to the original path. This part is the key ingredient
of the present work.

If the underlying path is not simple, the situation is more complicated since
extended signatures might not be able to recover the discrete route
$\mathbf{m}^{\varepsilon,\delta}$. The key to overcoming this issue
is to lift the path to its truncated signature path up to certain
high degree $N$ which is determined by the signature $g$. This is
achieved again through some stabilizing property. Here we will see that the
decay of signature plays a crucial role.

Now we present the basic functional setting on which our reconstruction
is based. Throughout the rest of the present paper, the roughness
$p\geqslant1$ is fixed. 

As the signature is invariant under reparametrization, it is natural to regard the parametrization-free tree-reduced path as a single object to be reconstructed in some functional space.

\begin{defn}
\label{def: reparametrization}A \textit{reparametrization} is a continuous
and strictly increasing function $\sigma:\ [0,1]\rightarrow[0,1],$
such that $\sigma(0)=0$ and $\sigma(1)=1.$ The group of reparametrizations
is denoted by $\mathcal{R}.$
\end{defn}

Let $(E,\rho)$ be a metric space, and $W=C([0,1];E)$ be the space
of continuous paths in $E.$ We introduce an equivalence relation
``$\sim$'' on $W$ by $x\sim x'$ if and only if $x_{\cdot}=x'_{\sigma(\cdot)}$
for some $\sigma\in\mathcal{R},$ and let $W^{\sim}$ be the corresponding
quotient space. Now define a distance function $d$ on $W^{\sim}$
by
\begin{equation}
d([x],[x'])=\inf_{\sigma\in\mathcal{R}}\sup_{t\in[0,1]}\rho\left(x_{t},x'_{\sigma(t)}\right),\ [x],[x']\in W^{\sim}.\label{eq: parametrization-free metric}
\end{equation}
It is not hard to see that $d$ is well-defined, symmetric and satisfies
the triangle inequality (c.f. \cite{BG2015} for a detailed discussion).
However, it is not positively definite in general. Now let $\Gamma$
be the set of paths $x\in W$ such that there exist $0\leqslant s<t\leqslant1$
with $x_{u}=x_{s}$ for all $u\in[s,t].$ If we restrict the relation
``$\sim$'' on $W_{0}=\Gamma^{c}\subset W$ and consider the corresponding
quotient space $W_{0}^{\sim}$, it is proved in \cite{BG2015} that
$d$ is indeed a metric on $W_{0}^{\sim}.$ 

In our situation, we take $E$ to be the space $E_{\lfloor p\rfloor}=T^{\lfloor p\rfloor}(\mathbb{R}^{d})$. Let $\mathbf{X}$
be a weakly geometric $p$-rough path. If $\mathbf{X}$ stays constant
during some $[s,t]$, then its signature path would not be
simple. It follows that every tree-reduced weakly geometric $p$-rough
path can be regarded as an element in $W_{0}$ (recall that $G^{\lfloor p\rfloor}(\mathbb{R}^{d})$
is canonically embedded in $E_{\lfloor p\rfloor}$). Moreover, from the uniqueness result we know that any two tree-reduced path with the same signature differ
by a reparametrization in the sense of Definition \ref{def: reparametrization}.
Therefore, the set of tree-reduced paths modulo reparametrization
can naturally be regarded as a subset $\mathcal{T}_{p}$ of $W_{0}^{\sim}.$
Our aim is to reconstruct the unique element in $\mathcal{T}_{p}$
corresponding to each given $g\in\mathfrak{S}_{p}$.

\section{Extended Signatures for Truncated Signature Paths}

The starting point of our reconstruction is the use of extended signatures
which are iterated integrals along ordered sequence of one forms.
In this section we present the basic ingredients for this part.

\subsection{The Signature of Truncated Signature Paths}

Since we aim at reconstructing the tree-reduced path $\mathbf{X}$
as a trajectory in $E_{\lfloor p\rfloor},$ we shall integrate along
one forms defined on $E_{\lfloor p\rfloor}$ rather than those on
$\mathbb{R}^{d}.$ This requires an important fact that $\mathbf{X}$
can be regarded as the first level path of some weakly geometric $p$-rough
path $\mathbf{Y}$ over $E_{\lfloor p\rfloor}$ whose signature, which
is an element of the tensor algebra over $E_{\lfloor p\rfloor}$,
is determined by the signature of $\mathbf{X}$ explicitly. Fortunately,
such $\mathbf{Y}$ can be constructed in a canonical way. A construction
based on ordered shuffles already appears in \cite{CDLL2015}, \cite{Lyons1998}
for the study of rough integration against weakly geometric rough
paths, and also in \cite{BGLY2014} for the uniqueness problem. Here
for completeness we present an equivalent construction which does
not rely on ordered shuffles and is more convenient to use from a
combinatorial point of view.

Let $\mathbf{X}$ be an arbitrary weakly geometric $p$-rough path.

Given $N\in\mathbb{N},$ let $y_{t}=(1,X_{0,t}^{1},\cdots,X_{0,t}^{N})$
be the truncated signature path of $\mathbf{X}$ up to degree $N$.
Then we know that $y$ takes values in the Euclidean space 
\[
E_{N}:=\bigoplus_{i=0}^{N}(\mathbb{R}^{d})^{\otimes i}.
\]
Therefore, the construction of the lifting of $y$ should take place
in the tensor algebra over $E_{N},$ where for each $n\geqslant1,$
we apply the following identification:
\[
(E_{N})^{\otimes n}\cong\bigoplus_{i_{1},\cdots,i_{n}=0}^{N}(\mathbb{R}^{d})^{\otimes(i_{1}+\cdots+i_{n})}.
\]
Equivalently we need to construct $Y_{s,t}^{n;i_{1},\cdots,i_{n}}\in(\mathbb{R}^{d})^{\otimes(i_{1}+\cdots+i_{n})}$
for $n\geqslant1,$ $0\leqslant i_{1},\cdots,i_{n}\leqslant N$ and
$(s,t)\in\Delta.$ $Y^{1}$ is just the increment of $y.$

Now we consider the case when $n=2.$ We first proceed by a formal
calculation to motivate the rigorous construction. Such calculation
is based on the formal differential equation 
\[
dS(\mathbf{X})_{s,t}=S(\mathbf{X})_{s,t}\otimes dx_{t}
\]
for the signature path (c.f. \cite{FV2010}, Proposition 7.8) together
with the shuffle product formula. Note that the identity in each step
below do not make sense when $p\geqslant2.$

If $1\leqslant i_{1},i_{2}\leqslant N,$ then formally we have
\begin{align*}
 & Y_{s,t}^{2;i_{1},i_{2}}\\
\doteq & \int_{s<u<t}Y_{s,u}^{1;i_{1}}\otimes dy_{u}^{i_{2}}\\
\doteq & \int_{s<u<t}(X_{0,u}^{i_{1}}-X_{0,s}^{i_{1}})\otimes dX_{0,u}^{i_{2}}\\
\doteq & \int_{s<u<t}\left(\sum_{j_{1}=1}^{i_{1}}X_{0,s}^{i_{1}-j_{1}}\otimes X_{s,u}^{j_{1}}\right)X_{0,u}^{i_{2}-1}\otimes dx_{u}\\
\doteq & \int_{s<u<t}\left(\sum_{j_{1}=1}^{i_{1}}X_{0,s}^{i_{1}-j_{1}}\otimes X_{s,u}^{j_{1}}\right)\left(\sum_{j_{2}=0}^{i_{2}-1}X_{0,s}^{i_{2}-1-j_{2}}\otimes X_{s,u}^{j_{2}}\right)\otimes dx_{u}\\
\doteq & \sum_{j_{1}=1}^{i_{1}}\sum_{j_{2}=0}^{i_{2}-1}\int_{s<u<t}P^{\tau(i_{1}-j_{1},j_{1},i_{2}-1-j_{2},j_{2})}(X_{0,s}^{i_{1}-j_{1}}\otimes X_{0,s}^{i_{2}-1-j_{2}}\\
 & \otimes X_{s,u}^{j_{1}}\otimes X_{s,u}^{j_{2}})\otimes dx_{u}
\end{align*}
\begin{align}
\doteq & \sum_{j_{1}=1}^{i_{1}}\sum_{j_{2}=0}^{i_{2}-1}\int_{s<u<t}P^{\tau(i_{1}-j_{1},i_{1},i_{2}-1-j_{2},j_{2})}\nonumber \\
 & \left(\sum_{\substack{\sigma_{1}\in\mathcal{S}(i_{1}-j_{1},i_{2}-1-j_{2})\\
\sigma_{2}\in\mathcal{S}(j_{1},j_{2})
}
}P^{\sigma_{1}\otimes\sigma_{2}}(X_{0,s}^{i_{1}+i_{2}-j_{1}-j_{2}-1}\otimes X_{s,u}^{j_{1}+j_{2}})\right)\otimes dx_{u}\nonumber \\
\doteq & \sum_{j_{1}=1}^{i_{1}}\sum_{j_{2}=0}^{i_{2}-1}\sum_{\substack{\sigma_{1}\in\mathcal{S}(i_{1}-j_{1},i_{2}-1-j_{2})\\
\sigma_{2}\in\mathcal{S}(j_{1},j_{2})
}
}\overline{P}^{\tau(i_{1}-j_{1},i_{1},i_{2}-1-j_{2},i_{2})}\nonumber \\
 & \circ\overline{P}^{\sigma_{1}\otimes\sigma_{2}}(X_{0,s}^{i_{1}+i_{2}-j_{1}-j_{2}-1}\otimes X_{s,u}^{j_{1}+j_{2}+1}).\label{eq: the second level of signature of signature}
\end{align}
Here $P^{\tau(i_{1}-j_{1},j_{1},i_{2}-1-j_{2},j_{2})},P^{\sigma_{1}\otimes\sigma_{2}}$
are linear transformations on $(\mathbb{R}^{d})^{\otimes(i_{1}+i_{2}-1)}$
defined by 
\[
P^{\tau(i_{1}-j_{1},j_{1},i_{2}-1,j_{2})}(a\otimes b\otimes c\otimes d)=a\otimes c\otimes b\otimes d
\]
for $a\in(\mathbb{R}^{d})^{\otimes(i_{1}-j_{1})},b\in(\mathbb{R}^{d})^{\otimes(i_{2}-1-j_{2})},c\in(\mathbb{R}^{d})^{\otimes j_{1}},d\in(\mathbb{R}^{d})^{\otimes j_{2}},$
and 
\[
P^{\sigma_{1}\otimes\sigma_{2}}(a\otimes b)=P^{\sigma_{1}}(a)\otimes P^{\sigma_{2}}(b)
\]
for $a\in(\mathbb{R}^{d})^{\otimes(i_{1}+i_{2}-j_{1}-j_{2}-1)},b\in(\mathbb{R}^{d})^{\otimes(j_{1}+j_{2})}$.
$\overline{P}^{\tau(i_{1}-j_{1},i_{1},i_{2}-1-j_{2},i_{2})},$ $\overline{P}^{\sigma_{1}\otimes\sigma_{2}}$
are their extensions to $(\mathbb{R}^{d})^{(i_{1}+i_{2})}$ by fixing
the last tensor component. The notation ``$\doteq$'' means that the
identities are formal. For simplicity, we can write (\ref{eq: the second level of signature of signature})
as
\begin{equation}
Y_{s,t}^{2;i_{1},i_{2}}\doteq\sum_{j=2}^{i_{1}+i_{2}}\sum_{\sigma\in\mathcal{A}(i_{1},i_{2},j)}P^{\sigma}(X_{0,s}^{i_{1}+i_{2}-j}\otimes X_{s,t}^{j}),\label{eq: definition of second level}
\end{equation}
where $\mathcal{A}(i_{1},i_{2},j)$ is a set of permutations of order
$i_{1}+i_{2}$ which fix the last component. Of course $\mathcal{A}(i_{1},i_{2},j)$
can be written down explicitly from (\ref{eq: the second level of signature of signature})
but we are not interested in its exact expression. The crucial point
here is that the right hand side of (\ref{eq: definition of second level})
is a well-defined element in $(\mathbb{R}^{d})^{\otimes(i_{1}+i_{2})},$
and it is reasonable to take it as the definition of $Y_{s,t}^{2;i_{1},i_{2}}$.

Inductively, assume that we have defined
\begin{equation}
Y_{s,t}^{n;i_{1},\cdots,i_{n}}=\sum_{j=n}^{i_{1}+\cdots+i_{n}}\sum_{\sigma\in\mathcal{A}(i_{1},\cdots,i_{n},j)}P^{\sigma}(X_{0,s}^{i_{1}+\cdots+i_{n}-j}\otimes X_{s,t}^{j})\label{eq: definition of n-th level}
\end{equation}
for all $1\leqslant i_{1},\cdots,i_{n}\leqslant N$ and $(s,t)\in\Delta,$
where $\mathcal{A}(i_{1},\cdots,i_{n},j)$ is a set of permutations
of order $i_{1}+\cdots+i_{n}$ which fix the last component. Then
for $1\leqslant i_{1},\cdots,i_{n+1}\leqslant N$, formally we have
\begin{eqnarray*}
 &  & Y_{s,t}^{n+1;i_{1},\cdots,i_{n+1}}\\
 & = & \int_{s<u<t}Y_{s,u}^{n;i_{1},\cdots,i_{n}}\otimes dy_{u}^{i_{n+1}}\\
 & = & \sum_{j=n}^{i_{1}+\cdots+i_{n}}\sum_{\sigma\in\mathcal{A}(i_{1},\cdots,i_{n},j)}\int_{s<u<t}\overline{P}^{\sigma}(X_{0,s}^{i_{1}+\cdots+i_{n}-j}\otimes X_{s,u}^{j}\otimes X_{0,u}^{i_{n+1}-1})\otimes dx_{u},
\end{eqnarray*}
where $\overline{P}^{\sigma}$ is the extension of $P^{\sigma}$ to
$(\mathbb{R}^{d})^{\otimes(i_{1}+\cdots+i_{n+1}-1)}$ by fixing the
last $i_{n+1}-1$ tensor components. By applying the same formal calculation
as in the case when $n=2,$ we obtain that 
\begin{equation}
Y_{s,t}^{n+1;i_{1},\cdots,i_{n},i_{n+1}}\doteq\sum_{j=n+1}^{i_{1}+\cdots+i_{n+1}}\sum_{\sigma\in\mathcal{A}(i_{1},\cdots,i_{n+1},j)}P^{\sigma}(X_{0,s}^{i_{1}+\cdots+i_{n+1}-j}\otimes X_{s,t}^{j}),\label{eq: definition of level n+1}
\end{equation}
where $\mathcal{A}(i_{1},\cdots,i_{n+1},j)$ is a set of permutations
of order $i_{1}+\cdots+i_{n+1}$ which fix the last component. We
take the right hand side of (\ref{eq: definition of level n+1}) as the definition
of $Y_{s,t}^{n+1;i_{1},\cdots,i_{n+1}}.$

Therefore, for each $n\geqslant1,$ $1\leqslant i_{1},\cdots,i_{n}\leqslant N$
and $(s,t)\in\Delta,$ we have defined $Y_{s,t}^{n;i_{1},\cdots,i_{n}}\in(\mathbb{R}^{d})^{\otimes(i_{1}+\cdots+i_{n})}$
from a purely algebraic point of view. In the case when one of those
$i_{j}$ equals zero, we simply let $Y_{s,t}^{n;i_{1},\cdots,i_{n}}=0.$
It remains to verify that this definition is exactly what we need.
Such verification is a straight forward consequence of the fact that
$\mathbf{X}$ is a geometric $p'$-rough path for every $p'\in(p,\lfloor p\rfloor+1)$
(c.f. \cite{FV2010}, Corollary 8.24).
\begin{prop}
\label{prop: signature of signature}Let $\mathbf{X}\in WG\Omega_{p}(\mathbb{R}^{d})$
and $N\in\mathbb{N}$. Then the formula (\ref{eq: definition of n-th level})
together with 
\[
Y_{s,t}^{n;i_{1},\cdots,i_{n}}=0,\ \mathrm{if\ one\ of\ }i_{j}=0,
\]
defines a multiplicative functional $\mathbb{Y}$ on the infinite
dimensional tensor algebra $T((E_{N}))$ over $E_{N}$ which takes
values in the group $G(E_{N})$ of exponential Lie series over $E_{N}.$
Moreover, for each $n\geqslant1,$ 
\begin{equation}
Y_{s,t}^{(n)}:=(1,Y_{s,t}^{1},\cdots,Y_{s,t}^{n}),\ (s,t)\in\Delta,\label{eq: truncated signature path of y}
\end{equation}
has finite $p$-variation in the sense of Definition \ref{def: p-rough path}.
Therefore, $Y^{(\lfloor p\rfloor)}$ defines a weakly geometric $p$-rough
path $\mathbf{Y}$ over $E_{N}$ whose signature path is $\mathbb{Y}$
with first level path being the truncated signature path of $\mathbf{X}$
up to degree $N.$

Finally, the signature of $\mathbf{Y}$ is determined by the signature
of $\mathbf{X}$ explicitly through the following formula:
\begin{eqnarray}
 &  & S(\mathbf{Y})_{0,1}^{n;i_{1},\cdots,i_{n}}\nonumber \\
 & = & \begin{cases}
\sum_{\substack{\sigma_{1}\in\mathcal{S}(i_{1},i_{2}-1)\\
\sigma_{2}\in\mathcal{S}(i_{1}+i_{2},i_{3}-1)\\
\cdots\\
\sigma_{n-1}\in\mathcal{S}(i_{1}+\cdots+i_{n-1},i_{n}-1)
}
}P^{\sigma_{n-1}\circ\cdots\circ\sigma_{1}}(X_{0,1}^{i_{1}+\cdots+i_{n}}), & 1\leqslant i_{1},\cdots,i_{n}\leqslant N,\\
0, & i_{j}=0\ \mathrm{for\ some}\ j,
\end{cases}\label{eq: signature of y}
\end{eqnarray}
where each shuffle $\sigma_{j}$ is regarded as a permutation of order
$i_{1}+\cdots+i_{n}$ by fixing the last $i_{j+1}+\cdots+i_{n}+1$
components.\end{prop}

\begin{proof}
Given $p'\in(p,\lfloor p\rfloor+1)$, let $x_{n}$ be a sequence of
continuous paths with bounded  variation whose lifting to $G^{\lfloor p\rfloor}(\mathbb{R}^{d})$
converges to $\mathbf{X}$ in $p'$-variation. The existence of $x_n$ is guaranteed by \cite{FV2010}, Corollary 8.24. Since the truncated
signature path $y_{n}$ of $x_{n}$ up to degree $N$ has bounded
 variation, it follows that each step in the previous formal calculation
becomes exact equality when $\mathbf{X}$ is replaced by $x_{n}.$
In other words, the signature path of $y_{n}$ given by iterated path
integrals against $y_{n}$ coincides with the definition given in
the previous formal construction. Therefore, it defines a multiplicative
functional on $T((E_{N}))$ satisfying the shuffle product formula. 

For each $n\geqslant1,$ since the free nilpotent group $G^{n}(E_{N})$
of step $n$ over $E_{N}$ is a closed subset of $T^{(n)}(E_{N}),$
by the continuity of signature (c.f. \cite{Lyons1998}, Theorem 2.2.2)
we conclude the first part of the proposition. Moreover, from the
explicit formula (\ref{eq: definition of n-th level}), it is easy
to see that $Y_{s,t}^{n;i_{1},\cdots,i_{n}}$ has the following estimate:
\[
|Y_{s,t}^{n;i_{1},\cdots,i_{n}}|\leqslant C(N,n,\omega(0,1))\cdot\omega(s,t)^{\frac{n}{p}},
\]
where 
\[
\omega(s,t):=\sum_{i=1}^{n}\sup_{\mathcal{P}_{[s,t]}}\sum_{l}\left|X_{t_{l-1},t_{l}}^{i}\right|^{\frac{p}{i}},\ (s,t)\in\Delta,
\]
and $C(N,n,\omega(0,1))$ is a constant depending only on $N,n,\omega(0,1).$
Therefore, $Y^{(n)}$ defined by (\ref{eq: truncated signature path of y})
has finite $p$-variation and the conclusion of the second part holds.
Finally, the formula (\ref{eq: signature of y}) for the signature
of $\mathbf{Y}$ can be seen directly by letting $s=0$ in the previous
formal calculation (in this case the permutation sets $\mathcal{A}(i_{1},\cdots,i_{n},j)$
can be written down easily).
\end{proof}

\subsection{Using Extended Signatures}

As mentioned before, we are going to reconstruct the tree-reduced
path by using extended signatures. In particular, we will integrate
along compactly supported one forms with continuous derivatives up
to order $\alpha:=\lfloor p\rfloor+1$. 

Let $\mathbf{X}$ be a weakly geometric $p$-rough path, and let $(\phi^{1},\cdots,\phi^{n})$
be a finite sequence of compactly supported $C^{\alpha}$-one forms
on $\mathbb{R}^{d}.$
\begin{defn}
The first level of the iterated path integral 
\[
\int_{0<t_{1}<\cdots<t_{n}<1}\phi^{1}(d\mathbf{X}_{t_{1}})\cdots\phi^{n}(d\mathbf{X}_{t_{n}})
\]
 is called the \textit{extended signature} of $\mathbf{X}$ along
$(\phi^{1},\cdots,\phi^{n})$, and it is denoted by $[\phi^{1},\cdots,\phi^{n}](x).$ 
\end{defn}

In general, the extended signature of $\mathbf{X}$ along $(\phi^{1},\cdots,\phi^{n})$
can either be interpreted through the approximation 
\begin{equation}
[\phi^{1},\cdots,\phi^{n}](x)=\lim_{k\rightarrow\infty}\int_{0<t_{1}<\cdots<t_{n}<1}\phi^{1}(dx_{t_{1}}^{(k)})\cdots\phi^{n}(dx_{t_{n}}^{(k)}),\label{eq: extended signature via approximation}
\end{equation}
where $x^{(k)}$ is a sequence of continuous paths with bounded variation
whose lifting to $G^{\lfloor p\rfloor}(\mathbb{R}^{d})$ converges
to $\mathbf{X}$ in $p'$-variation for some $p'\in(p,\lfloor p\rfloor+1)$, or through the unique solution to the rough differential equation
\[
\begin{cases}
dx_{t}^{i}=dx_{t}^{i}, & 1\leqslant i\leqslant d,\\
dy_{t}^{j}=y_{t}^{j-1}\sum_{i=1}^{d}\phi_{i}^{j}(x_{t})dx_{t}^{i}, & 1\leqslant j\leqslant n,\\
x_{0}=0,y_{0}=0,
\end{cases}
\]
where $y_{t}^{0}:=1.$

\begin{rem}
In the notation $[\phi^{1},\cdots,\phi^{n}](x)$, we have used
small ``$x$'', which denotes the first level path of $\mathbf{X}$,
to emphasize that the integral is essentially constructed on $\mathbb{R}^{d}$
although the rigorous definition relies on the rough path nature of
$\mathbf{X}$.
Later on we will use extended signatures constructed on $E_{N}$ along
one forms over $E_{N}$ for $N\geqslant\lfloor p\rfloor.$
\end{rem}

Let $g$ be the signature of $\mathbf{X}$ and let $\phi$ be a compactly
supported $C^{\alpha}$-one form. The starting point of our reconstruction is the crucial fact that the integral $\int_{0}^{1}\phi(d\mathbf{X}_{t})$ can be explicitly
reconstructed from the knowledge of $g$ and $\phi.$

Indeed, let $n_{0}\geqslant1$ be such that $B_{n_{0}}=\{x\in\mathbb{R}^d:\ |x|\leqslant n_0\}$ contains the support
of $\phi.$ For each $n\geqslant n_{0},$ let $p_{n}$ be a polynomial
one form such that 
\[
\sup_{0\leqslant j\leqslant\alpha}\sup_{B_{n}}\left|D^{j}(\phi-p_{n})\right|\leqslant\frac{1}{n}.
\]
The existence of $p_{n}$ is guaranteed by the work of Bagby-Bos-Levenberg
\cite{BBL2002}, Theorem 1. Moreover, from their proof the construction
of $p_{n}$ is explicit. 

Since $p_{n}$ is a polynomial one form, the integral $\int_{0}^{1}p_{n}(d\mathbf{X}_{t})$
can be directly computed from $g$ using the shuffle
product formula. Now it remains to show the following simple fact.
\begin{prop}
The value of the integral $\int_{0}^{1}\phi(d\mathbf{X}_{t})$ is
given by 
\[
\int_{0}^{1}\phi(d\mathbf{X}_{t})=\lim_{n\rightarrow\infty}\int_{0}^{1}p_{n}(d\mathbf{X}_{t}).
\]
\end{prop}
\begin{proof}
Since $x([0,1])$ is a compact subset of $\mathbb{R}^{d},$ we know
that $x([0,1])\subset\subset B_{n_{1}}$ for some $n_{1}\geqslant n_{0}$.
Therefore,
\[
\sup_{0\leqslant j\leqslant\alpha}\sup_{B_{n_{1}}}\left|D^{j}(\phi-p_{n})\right|<\frac{1}{n},\ \forall n\geqslant n_{1}.
\]
The result follows from the continuity of the integration map $\phi\mapsto\int_{0}^{1}\phi(d\mathbf{X}_{t})$
(c.f. \cite{FV2010}, Theorem 10.47). Note that here the path $\mathbf{X}$
is fixed, thus the continuity holds under the $C^{\alpha}$-norm over
$B_{n_{1}}.$ 
\end{proof}

The same argument applies to extended signatures. In particular, for
a given finite sequence $(\phi^{1},\cdots,\phi^{n})$ of compactly
supported $C^{\alpha}$-one forms, the extended signature $[\phi^{1},\cdots,\phi^{n}](x)$
can be reconstructed from the knowledge of $g$ and these one forms.

In our reconstruction problem, as we have pointed out before, we shall
integrate over $E_{\lfloor p\rfloor}$ rather than over $\mathbb{R}^{d}.$
More generally, given $N\geqslant\lfloor p\rfloor$, let $y$ be the
truncated signature path of $\mathbf{X}$ up to some degree $N$,
and let $\mathbf{Y}$ be the weakly geometric $p$-rough path over
$E_{N}$ defined in Proposition \ref{prop: signature of signature}.
We know that the signature of $\mathbf{Y}$ is determined by $g$
through the formula (\ref{eq: signature of y}). Therefore, given
a finite sequence $(\Phi^{1},\cdots,\Phi^{n})$ of compactly supported
$C^{\alpha}$-one forms on $E_{N},$ the extended signature $[\Phi^{1},\cdots,\Phi^{n}](y)$
can be reconstructed from the knowledge of $g$ and these one forms
on $E_{N}.$

\section{The Reconstruction: Non-self-intersecting Case}

In this section, we develop our reconstruction for the case when the
tree-reduced weakly geometric $p$-rough paths are simple. Although
we could treat the general case in one go, a good understanding
of the non-self-intersecting case is very helpful.

Recall that 
\[
E_{\lfloor p\rfloor}=\bigoplus_{i=0}^{\lfloor p\rfloor}(\mathbb{R}^{d})^{\otimes i}
\]
is an Euclidean space of dimension $D=1+d+\cdots+d^{\lfloor p\rfloor}$
with Euclidean norm given by 
\[
\|\cdot\|_{E_{\lfloor p\rfloor}}=\sum_{i=0}^{\lfloor p\rfloor}\|\cdot\|_{(\mathbb{R}^{d})^{\otimes i}}.
\]
Elements in $E_{\lfloor p\rfloor}$ are of the form $a=(a^{I})_{0\leqslant|I|\leqslant\lfloor p\rfloor}$,
where $I$ is a word over the alphabet $\{1,\cdots,d\}$. Here the order
of coordinates $a^{I}$ in a homogeneous tensor product $(\mathbb{R}^{d})^{\otimes k}$
is not important, but the increasing order with respect to the tensor
degree is important (in order to make use of the decay of signature
as we will see in Section 6).

In the rest of this section, we assume that $\mathbf{X}$ is a simple
tree-reduced weakly geometric $p$-rough path with signature $g.$

\subsection{Recovering the Discrete Route in a Given Geometric Scheme}

Let $\{K_{0},\cdots,K_{r}\}$ be a finite family of bounded domains
in $E_{\lfloor p\rfloor}$ whose closures are mutually disjoint. Suppose
that the path $\mathbf{X}$ starts in $K_{0}$.

We define $L$ to be the total number of domains $K_{i}$ visited
by $\mathbf{X}$ in order (excluding the initial one $K_{0}$). $L$
is finite from the continuity of $\mathbf{X}$ and the disjointness
of those $\overline{K_{i}}$. Let 
\[
\{(\tau_{k},m_{k})\in[0,1]\times\{0,\cdots,r\}:\ 0\leqslant k\leqslant L\}
\]
be the corresponding sequence of entry times together with domains
visited, where $\tau_{0}=0$ and $m_{0}=0.$ Here we label the domains
by their subscripts $\{0,\cdots,r\}.$ 

We should point out that revisiting the same domain before
entering other ones does not count, but revisiting after entering
some other domain does count. Moreover, we only consider entrance of these domains (which are open) but do not
consider the case of bouncing at the boundaries. 

Figure \ref{fig: new figure 1} illustrates the notions in dimension $2$. In this example $L=7$, and the discrete route of the underlying path is given by the 
word $(0,3,4,5,0,1,0,3).$

\begin{figure} 
\begin{center} 
\includegraphics[scale=0.35]{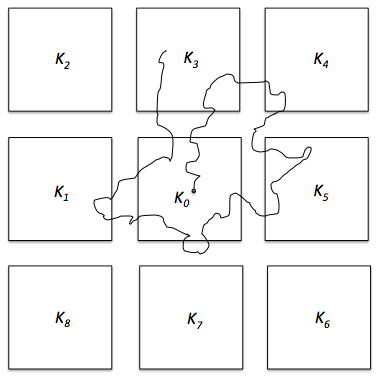}
\protect\caption{This figure illustrates the construction of the discrete route of a path in dimension $2$.
In this example $L=7$, and the discrete route of the underlying path is given by the 
word $(0,3,4,5,0,1,0,3).$}\label{fig: new figure 1}
\end{center}
\end{figure}

The next step is to show that in the non-self-intersecting case, the
ordered sequence $(m_{0},\cdots,m_{L})$ of domains visited by $\mathbf{X}$
can be reconstructed from the signature $g$ through computing extended
signatures. 

To be more precise, we have the following result. The construction
of the one forms involved is motivated from the work of \cite{BGLY2014}. 
\begin{prop}
\label{prop: existence of one forms} $Let$ $S=\{m_{0},\cdots,m_{L}\}\subset\{0,\cdots,r\}$
be the associated set of domains visited by $\mathbf{X}.$

(1) For each $m\in S,$ there exists a $C^{\alpha}$-one form $\Phi_{m}$
supported on $\overline{K_{m}},$ such that the extended signature
\[
\left[\Phi_{m_{0}},\cdots,\Phi_{m_{L}}\right](\mathbf{X})\neq0.
\]
Here we should regard $\mathbf{X}$ as the first level of the rough
path $\mathbf{Y}$ in Proposition \ref{prop: signature of signature}
with $N=\lfloor p\rfloor,$ and the extended signature is constructed
on $E_{\lfloor p\rfloor}$ for $\mathbf{Y}$ (see the last paragraph
of Section 4). 

(2) For $0\leqslant m\leqslant r,$ let $\Psi_{m}$ be any given $C^{\alpha}$-one
form supported on $\overline{K_{m}}$. Suppose that $\mathbf{n=}(n_{0}=0,n_{1},\cdots,n_{l})$
is an arbitrary word over the alphabet $\{0,\cdots,r\}$ such that
$n_{k}\neq n_{k-1}$ for $1\leqslant k\leqslant l$. If $n_{k}\notin S$
for some $1\leqslant k\leqslant l$, or if it is different from the
word $\mathbf{m}=(m_{0},\cdots,m_{L})$ when $l\geqslant L$, then
the extended signature 
\[
[\Psi_{n_{0}},\cdots,\Psi_{n_{l}}](\mathbf{X})=0.
\]
\end{prop}
\begin{proof}
Given $m\in S,$ let $0\leqslant k\leqslant L$ be such that $m_{k}=m$.
By the definition of $\tau_{k}$, there exist $\tau_{k}<s<t<\tau_{k+1}$
such that $\mathrm{Im}(\mathbf{X}|_{[s,t]})\subset K_{m}$ (here we
set $\tau_{L+1}=1$). Let $s',t'$ satisfy $s<s'<t'<t.$ Since $\mathbf{X}$
is simple, we know that 
\[
\mathrm{Im}(\mathbf{X}|_{[s,s']})\bigcap\mathrm{Im}(\mathbf{X}|_{[t',t]})=\emptyset
\]
and 
\[
\mathrm{Im}(\mathbf{X}|_{[s',t']})\bigcap\mathrm{Im}(\mathbf{X}|_{[\tau_{k},s]\cup[t,\tau_{k+1}]})=\emptyset.
\]
It follows that there exist open neighborhoods $U_{1},U_{2},V_{1},V_{2}$
of $\mathrm{Im}(\mathbf{X}|_{[s,s']}),$ $\mathrm{Im}(\mathbf{X}|_{[t',t]}),$
$\mathrm{Im}(\mathbf{X}|_{[s',t']}),$ $\mathrm{Im}(\mathbf{X}|_{[\tau_{k},s]\cup[t,\tau_{k+1}]})$
respectively, such that 
\begin{equation}
U_{1}\bigcup V_{1}\bigcup U_{2}\subset\subset K_{m}\label{eq: neighborhoods in K_m}
\end{equation}
and
\[
U_{1}\bigcap U_{2}=V_{1}\bigcap V_{2}=\emptyset.
\]
Let $F,G$ be two $C^{\alpha}$-functions on $E_{\lfloor p\rfloor}$
such that 
\[
F=0\ \mathrm{on}\ U_{1},\ F=1\ \mathrm{on}\ U_{2},
\]
and 
\[
G=0\ \mathrm{on}\ V_{2},\ G=1\ \mathrm{on}\ V_{1},
\]
respectively. We define $\Phi_{k}=GdF.$ From the construction it
is straight forward to see that
\[
\int_{\tau_{k}}^{\tau_{k+1}}\Phi_{k}(d\mathbf{Y}_{t})=F(\mathbf{X}_{t'})-F(\mathbf{X}_{s'})=1.
\]
Moreover, the value of the integral depends only on the definition
of $\Phi_{k}$ on $U_{1}\cup U_{2}\cup V_{1}\cup V_{2}$. Together
with the fact that $\Phi_{k}=0$ on $V_{2}$ and (\ref{eq: neighborhoods in K_m}),
we can certainly modify the definition of $\Phi_{k}$ without changing
its values on $U_{1}\cup U_{2}\cup V_{1}\cup V_{2}$ so that it is
supported on $\overline{K_{m}}$. Furthermore, it follows again from
the simpleness of $\mathbf{X}$ that the family 
\[
\left\{ \mathrm{Im}(\mathbf{X}|_{[\tau_{k},\tau_{k+1}]}):\ 0\leqslant k\leqslant L\ \mathrm{with}\ m_{k}=m\right\} 
\]
are mutually disjoint. Therefore, by choosing $U_{1}\cup U_{2}\cup V_{1}\cup V_{2}$
small enough, we can define a single $C^{\alpha}$-one form $\Phi_{m}$
supported on $\overline{K_{m}},$ such that 
\[
\int_{\tau_{k}}^{\tau_{k+1}}\Phi_{m}(d\mathbf{Y}_{t})=1
\]
for every $0\leqslant k\leqslant L$ satisfying $m_{k}=m.$ The first
part of the lemma follows from the decomposition of the extended signature:
\[
\left[\Phi_{m_{0}},\cdots,\Phi_{m_{L}}\right](\mathbf{X})=\prod_{k=0}^{L}\int_{\tau_{k}}^{\tau_{k+1}}\Phi_{m_{k}}(d\mathbf{Y}_{t})
\]
given in \cite{BG2015}, Lemma 5.1 (1).

The second part follows from \cite{BG2015}, Lemma 4.2 and Lemma 5.1
(2), (3). 
\end{proof}

\begin{rem}
The second part of Proposition \ref{prop: existence of one forms}
does not depend on the non-self-intersecting assumption; it is true
for all weakly geometric $p$-rough paths.
\end{rem}
The construction of the one forms in Proposition \ref{prop: existence of one forms}
certainly depends on the trajectory of $\mathbf{X}$ in a crucial
way. However, once existence is guaranteed by the result of Proposition
\ref{prop: existence of one forms}, we actually do not need the specific
construction and it is sufficient to compute extended signatures with
respect to a pre-specified countable generating set for the space
of one forms. This is a consequence of separability and continuity.

Firstly, as $\alpha=\lfloor p\rfloor+1$ is a positive integer, it is well known that the space of $C^{\alpha}$-functions
supported on some given compact set $K$ equipped with the $C_{K}^{\alpha}$-topology
is separable. Therefore, for each domain $K_{i},$ we can specify
a countable dense subset $\{\Phi_{n}^{(i)}:\ n\geqslant1\}$ of the
space of $C^{\alpha}$-one forms supported on $\overline{K_{i}}$.
If the geometry of $K_{i}$ is simple, the construction of $\Phi_{n}^{(i)}$
can be made explicit for instance by using wavelets. In fact in our
situation the domains $K_{i}$ are just cubes or convex hulls of two
concentric cubes. 

Secondly, for given $l\geqslant0$, let $\mathcal{W}_{l}$ be the
set of words $\mathbf{n}=(n_{0}=0,n_{1},\cdots,n_{l})$ such that
$n_{k}\neq n_{k-1}$ for $1\leqslant k\leqslant l$. Given a word
$\mathbf{n}\in\mathcal{W}_{l},$ let $\{I_{m}(g;l,\mathbf{n}):\ m\geqslant1\}$
be an enumeration of all possible extended signatures of $\mathbf{Y}$
along the word $\mathbf{n}$ with respect to the previous specification
of generating one forms. Define
\begin{equation}
\chi(g;l,\mathbf{n})=\begin{cases}
1, & \mathrm{if}\ I_{m}(g;l,\mathbf{n})\neq0\ \mathrm{for\ some\ }m;\\
0, & \mathrm{otherwise.}
\end{cases}\label{eq: defining chi}
\end{equation}
Note that $\chi(g;l,\mathbf{n})$ is determined by the signature $g$
and the word $\mathbf{n}\in\mathcal{W}_{l}$. According to the continuity
of rough path integrals along one forms, a direct consequence of Proposition
\ref{prop: existence of one forms} is the following.
\begin{cor}
\label{cor: reconstructing the visiting sequence}The total number
$L$ of domains visited by $\mathbf{X}$ is given by 
\[
L=\sup\left\{ l\geqslant0:\ \chi(g;l,\mathbf{n})=1\ \mathrm{for\ some\ }\mathbf{n}\in\mathcal{W}_{l}\right\} ,
\]
Moreover, there is one and only one word $\mathbf{n}=(n_{0},\cdots,n_{L})\in\mathcal{W}_{L}$
such that $\chi(g;l,\mathbf{n})=1,$ and $\mathbf{n}$ is exactly
the word $\mathbf{m}=(m_{0},\cdots,m_{L})$ corresponding to the ordered
sequence of domains visited by $\mathbf{X}$.
\end{cor}

In other words, in the case when $\mathbf{X}$ is simple, the word
$\mathbf{m}$ can be reconstructed from the signature of $\mathbf{X}$. 

\begin{rem}
In contrast to the probabilistic setting, here as we have to treat every
path equally, a universal construction of a single one form for
all paths is not possible. Instead we need to compute extended signatures
along all possible ``directions'' to reveal geometric information
about the underlying path. This is the nature of the problem if we aim
at finding a universal reconstruction for the class of all tree-reduced
weakly geometric $p$-rough paths in one go.
\end{rem}

\subsection{The Key Ingredient: A Stable Quantity on Words}

Now we turn to the reconstruction of $\mathbf{X}$ by constructing
our geometric scheme in a more specific way.

Recall that $E_{\lfloor p\rfloor}$ is an Euclidean space of dimension
$D.$ For $0\leqslant j\leqslant D,$ let $V_{j}$ be the set of points
$z=(z^{I})_{0\leqslant|I|\leqslant\lfloor p\rfloor}\in E_{\lfloor p\rfloor}$
such that exactly $j$ of those $z^{I}$ are integers and the rest
are half-integers (i.e. of the form $z^{I}=n/2$ where $n$ is an
odd integer). 
\begin{defn}
For $0\leqslant j\leqslant D,$ the \textit{open} $j$\textit{-skeleton
}$K_{j}$ is defined to be the disjoint union 
\[
K_{j}=\bigcup_{z\in V_{j}}\{a\in E_{\lfloor p\rfloor}:\ \left|a^{I}-z^{I}\right|<\frac{1}{2}\ \mathrm{if\ }z^{I}\ \mathrm{is\ an\ integer\ and}\ a^{I}=z^{I}\ \mathrm{otherwise}\}
\]
of $j$-dimensional open faces with centers in $V_{j},$ and the \textit{closed
$j$-skeleton $C_{j}$ }is defined to be 
\[
C_{j}=\bigcup_{i=0}^{j}K_{i}.
\]

\end{defn}

Figure \ref{fig: new figure 2} illustrates the definition of the open and closed skeletons when $D=2$. Here $K_2$ consists of the set of open cubes centered at integer points, $K_1$ consists of their open edges and $K_0$ is the set of vertices.

\begin{figure} 
\begin{center} 
\includegraphics[scale=0.35]{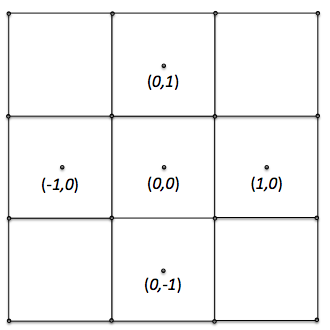}
\protect\caption{This figure illustrates the definition of the open and closed skeletons in when $D=2$. Here $K_2$ consists of the set of open cubes centered at integer points, $K_1$ consists of their open edges and $K_0$ is the set of vertices.}\label{fig: new figure 2}
\end{center}
\end{figure}

From now on, we fix $\varepsilon>0$ for the rest of this subsection,
and for simplicity in the upcoming notions the dependence on $\varepsilon$
will be omitted. Correspondingly, the skeletons $K_j$ and $C_j$ are scaled to the $\varepsilon$ order.

Let $\delta$ be a parameter with $\delta<\varepsilon$. For convenience
we should think of $\varepsilon,\delta$ as discrete parameters (for
instance we can simply take $\varepsilon=1/m,$ $\delta=1/n$ with
$n>m$). The same remark applies to other parameters to be introduced
later on.

For $z\in V_{D},$ let $H_{z}^{\delta;D}$ be the cube defined by
\[
H_{z}^{\delta;D}=\left\{ a\in E_{\lfloor p\rfloor}:\ \left|a^{I}-\varepsilon z^{I}\right|<\frac{\varepsilon-\delta}{2}\right\} ,
\]
The geometric scheme of the collection $\{\mathbf{1}+H_{z}^{\delta;D}:\ z\in V_{D}\}$
($\mathbf{1}=(1,0,\cdots,0)$ is the unit of $E_{\lfloor p\rfloor}$)
is denoted by $\mathcal{C}^{\delta;D}.$ In other words, $\mathcal{C}^{\delta;D}$
consists of cubes with edge length $\varepsilon-\delta$ and narrow
tunnels with width $\delta.$

We define $L^{\delta;D}$ and $\{\tau_{k}^{\delta;D},m_{k}^{\delta;D}:\ 0\leqslant k\leqslant L^{\delta;D}\}$
for the geometric scheme $\mathcal{C}^{\delta;D}$ in the same way
as in the last subsection. Here we label the domains in $\mathcal{C}^{\delta;D}$ by elements in $V_D$ so that $m_k^{\delta;D}\in V_D$ is the center of the $k$-th visited cube. According to the previous discussion, the word 
\[
\mathbf{m}^{\delta;D}=(m_{0}^{\delta;D},\cdots,m_{L^{\delta;D}}^{\delta;D})
\]
can be reconstructed from the signature $g$. Associated with the
word $\mathbf{m}^{\delta;D},$ we can construct a polygonal path by
connecting the centers $m_{k}^{\delta;D}$ ($0\leqslant k\leqslant L^{\delta;D}$)
in order. 

\begin{rem}
Although $\mathcal{C}^{\delta;D}$ contains countably many cubes,
the discussion in the last subsection certainly carries through in
without any difficulty since from compactness we know
that the underlying path can visit at most finitely many different
cubes.
\end{rem}

As we have pointed out in the discussion in Section
3, we cannot expect the convergence of the polygonal approximation to our underlying path in any sense when we let $\varepsilon,\delta\rightarrow0$. The main idea of overcoming this issue is to let the signature $g$
choose the ``right'' width $\delta(g)$, so that in the resulting
geometric scheme $\mathcal{C}^{\delta(g);D}$, we can conclude that
if the underlying path $\mathbf{X}$ has a long excursion in the complement of $\mathcal{C}^{\delta(g);D}$ (i.e. the tunnels) during some time period $[s,t]$, it
has to spend some time period $[s',t']\subset[s,t]$ having such a long excursion \textit{in the closed} $(D-1)$-\textit{skeleton}. This enables us
to start a recursive construction until we reach the $0$-skeleton by
adding more and more domains over the open skeleton in each dimension.
The conclusion is that in the final geometric scheme $\mathcal{C}$, the path is not able to have a long excursion in the complement of $\mathcal{C}$ and the convergence is then immediate. 

The way of choosing $\delta$ by the signature is through a stable
quantity on words.

Let $\mathcal{W}^{D}$ be the set of all finite words $\mathbf{z}=(z_{0}=0,z_{1},\cdots,z_{n})$
over the alphabet $V_{D}$ such that $z_{k}\neq z_{k-1}$ for all
$k.$
\begin{defn}
A \textit{stable quantity }over $V_{D}$ with respect to the geometric
schemes $\mathcal{C}^{\delta;D}$ ($0<\delta<\varepsilon$) is a map
$\mathfrak{s}:\ \mathcal{W}^{D}\rightarrow\mathbb{Z},$ such that
for any continuous path $x$ in $E_{\lfloor p\rfloor}$ starting at
$\mathbf{1},$ the limit 
\begin{equation}
\lim_{\delta\rightarrow0}\mathfrak{s}(\mathbf{m}^{\delta;D})\label{eq: the limit for the stablizing quantity}
\end{equation}
always exists, where $\mathbf{m}^{\delta;D}$ is the word corresponding
to the ordered sequence of cubes visited by $x$ in the geometric
scheme $\mathcal{C}^{\delta;D}$ defined as before.
\end{defn}

\begin{example}
Given a word $\mathbf{z}=(z_{0},\cdots,z_{n})\in\mathcal{W}^{D},$
let $S_{\mathbf{z}}=\{z_{0},\cdots,z_{n}\}\subset V_{D}$ be the associated
set of letters in $\mathbf{z}$. Define $\mathfrak{s}(\mathbf{z})=\sharp(S_{\mathbf{z}})$
(the total number of elements in $S_{\mathbf{z}}$). Then $\mathfrak{s}$
defines a stable quantity. Indeed, let $x$ be a continuous path $x$
starting at $\mathbf{1}.$ For $\delta_{1}<\delta_{2},$ it is easy
to see that $\mathbf{m}^{\delta_{2};D}$ is a subword of $\mathbf{m}^{\delta_{1};D}$,
and thus $\mathfrak{s}(\mathbf{m}^{\delta_{2};D})\leqslant\mathfrak{s}(\mathbf{m}^{\delta_{1};D}).$
Moreover, from the boundedness of $x$ we know that 
\[
\sup_{0<\delta<\varepsilon}\mathfrak{s}(\mathbf{m}^{\delta;D})<\infty.
\]
Therefore the limit (\ref{eq: the limit for the stablizing quantity})
exists. However, this stable quantity is not applicable for our purpose.
\end{example}

\begin{example}
For $\mathbf{z}=(z_{0},\cdots,z_{n})\in\mathcal{W}^{D},$ define $\mathfrak{s}(\mathbf{z})=n$
(the length of $\mathbf{z}$). Then $\mathfrak{s}$ is not a stable
quantity. This is simply because we can construct a path $x$, such
that the length of the associated word $\mathbf{m}^{\delta;D}$ explodes
as $\delta\rightarrow0$ (think of the topologist's sine curve in
higher dimensions and note that the distance between neighboring cubes
tends to zero).
\end{example}

Now we are going to define a stable quantity $\mathfrak{s}_{D},$
such that the width parameter 
\begin{equation}
\delta_{1}:=\frac{1}{2}\sup\{0<\delta<\varepsilon:\ \mathfrak{s}_{D}(\mathbf{m}^{\delta';D})=\mathrm{const.}\ \forall\delta'\leqslant\delta\}>0\label{eq: choice of delta_1}
\end{equation}
obtained from the stability of $\mathfrak{s}_{D}(\mathrm{m}^{\delta;D})$
as $\delta\rightarrow0$ will serve our purpose. In particular we
know that $\mathfrak{s}_{D}(\mathbf{m}^{\delta';D})=\mathfrak{s}_{D}(\mathbf{m}^{\delta_{1};D})$
for all $\delta'\leqslant\delta_{1}.$ Note that $\delta_{1}$ is
determined by the signature $g$ explicitly. 
\begin{defn}
Let $\mathbf{z}=(z_{0},\cdots,z_{n})\in\mathcal{W}^{D}.$ An \textit{admissible
chain} $c$ in $\mathbf{z}$ is a subword $(z_{i_{1}},\cdots,z_{i_{r}})$
for some $0\leqslant i_{1}<\cdots<i_{r}\leqslant n$ such that 
\[
\left|z_{i_{k}}-z_{i_{k-1}}\right|\geqslant2\sqrt{D}
\]
for all $2\leqslant k\leqslant r.$ 
\end{defn}

For $\mathbf{z}\in\mathcal{W}^{D},$ we define $\mathfrak{s}_{D}(\mathbf{z})$
to be the maximal length $r$ of all possible admissible chains in
$\mathbf{z}$ (if admissible chains do not exist, we simply define
$\mathfrak{s}_{D}(\mathbf{z})=1$). It is obvious that $\mathfrak{s}_{D}(\mathbf{z})$
is well-defined. Any admissible chain in $\mathbf{z}$ with length
$\mathfrak{s}_{D}(\mathbf{z})$ is called a \textit{maximal }admissible
chain. Note that maximal admissible chains may not be unique.
\begin{example}
Consider the case when $D=2.$ Let 
\[
\mathbf{z}=((0,0),(0,1),(1,1),(2,1),(2,2),(3,1)).
\]
 Then $\mathfrak{s}_{D}(\mathbf{z})=2,$ and $((0,0),(2,2)),$ $((0,0),(3,1)),$
$((0,1),(3,1))$ are all maximal admissible chains.
\end{example}

Now we have the following result.
\begin{lem}
\label{lem: proof of stability}$\mathfrak{s}_{D}$ is a stable quantity.\end{lem}
\begin{proof}
Let $x$ be a continuous path starting at $\mathbf{1}.$ For $0<\delta<\varepsilon,$
let $c^{\delta}=(z_{i_{1}},\cdots,z_{i_{r}})$ be an admissible chain
of $\mathbf{m}^{\delta;D}$ such that $r=\mathfrak{s}_{D}(\mathbf{m}^{\delta;D}).$
It follows that $c^{\delta}$ is also an admissible chain of $\mathbf{m}^{\delta';D}$
for any $\delta'<\delta$ since $\mathbf{m}^{\delta;D}$ is a subword
of $\mathbf{m}^{\delta';D}$. Therefore, $\mathfrak{s}_{D}(\mathbf{m}^{\delta';D})\geqslant\mathfrak{s}_{D}(\mathbf{m}^{\delta;D}).$
On the other hand, since $x$ is continuous, there exists $\eta>0$
such that 
\[
|x_{t}-x_{s}|\geqslant\sqrt{D}\varepsilon\implies|t-s|\geqslant\eta.
\]
In the geometric scheme $\mathcal{C}^{\delta;D},$ let $\zeta_{1}<\cdots<\zeta_{r}$
be a sequence of entry times corresponding to the admissible chain
$c^{\delta}$ (they exist by definition). Then for each $k$ we have
\begin{eqnarray*}
\left|x_{\zeta_{k}}-x_{\zeta_{k-1}}\right| & \geqslant & \left|\varepsilon z_{i_{k}}-\varepsilon z_{i_{k-1}}\right|-\left|x_{\zeta_{k}}-(\mathbf{1}+\varepsilon z_{i_{k}})\right|-\left|x_{\zeta_{k-1}}-(\mathbf{1}+\varepsilon z_{i_{k-1}})\right|\\
 & \geqslant & 2\sqrt{D}\varepsilon-\frac{\sqrt{D}}{2}\varepsilon-\frac{\sqrt{D}}{2}\varepsilon\\
 & = & \sqrt{D}\varepsilon.
\end{eqnarray*}
It follows that $|\zeta_{k}-\zeta_{k-1}|\geqslant\eta$ for every
$k$, and thus $r\leqslant1/\eta+1.$ Therefore, the limit (\ref{eq: the limit for the stablizing quantity})
exists.
\end{proof}

Now we define $\delta_{1}$ by the formula (\ref{eq: choice of delta_1}). 

The next step is to develop similar construction for geometric schemes
over the open $(D-1)$-skeleton.

Suppose $0<\delta\leqslant\delta_{1}.$ For $z\in V_{D-1}$, define
\begin{eqnarray}
H_{z}^{\delta;D-1} & = & \left\{ a\in E_{\lfloor p\rfloor}:\ \left|a^{I}-\varepsilon z^{I}\right|<\frac{\varepsilon-\delta}{2}\ \mathrm{if\ }z^{I}\ \mathrm{is\ an\ integer}\right.\nonumber \\
 &  & \left.\mathrm{and}\ \left|a^{I}-\varepsilon z^{I}\right|<\frac{\delta}{4}\ \mathrm{otherwise}\right\} .\label{eq: constructing the cube}
\end{eqnarray}
Let $C_{z}^{\delta;D-1}$ be the convex hull of the two cubes $H_{z}^{\delta;D-1}$
and $H_{z}^{\delta_{1};D-1}.$ We denote the geometric scheme of $\{\mathbf{1}+C_{z}^{\delta;D-1}:\ z\in V_{D-1}\}$
by $\mathcal{C}^{\delta;D-1}.$ Figure \ref{figure 1} (excluding the $4$ cubes) illustrates the construction of $\mathcal{C}^{\delta;D-1}$ when $D=2$.

\begin{figure} 
\begin{center} 
\includegraphics[scale=0.36]{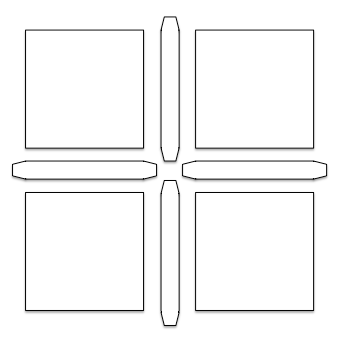}
\protect\caption{This figure illustrates the construction of $\mathcal{C}^{\delta;D-1}$ and $\mathcal{C}^\varepsilon(g)$ (to be defined later) when $D=2$.}\label{figure 1}
\end{center}
\end{figure}

\begin{lem}
\label{lem: basic feature of geometric schemes}(1) For each $\delta\leqslant\delta_{1},$
$\{\overline{C_{z}^{\delta;D-1}}:\ z\in V_{D-1}\}$ are mutually disjoint.

(2) For each $z\in V_{D-1},$ if $\delta'<\delta\leqslant\delta_{1},$
then $\overline{C_{z}^{\delta;D-1}}\subset\overline{C_{z}^{\delta';D-1}}.$\end{lem}
\begin{proof}
(1) Assume that $a\in\overline{C_{z}^{\delta;D-1}}\cap\overline{C_{z'}^{\delta;D-1}}$
for some $z\neq z'\in V_{D-1}.$ Let $\mathcal{I},\mathcal{I}'$ be
the sets of components of $z,z'$ which are integers respectively.
If $\mathcal{I}=\mathcal{I}'$, then for any $I$ we have 
\[
\left|a^{I}-\varepsilon z^{I}\right|\leqslant\frac{\varepsilon-\delta}{2},\ \left|a^{I}-\varepsilon(z')^{I}\right|\leqslant\frac{\varepsilon-\delta}{2}
\]
if $I\in\mathcal{I}$, which implies that $z^{I}=(z')^{I},$ and 
\[
\left|a^{I}-\varepsilon z^{I}\right|\leqslant\frac{\delta}{4},\ \left|a^{I}-\varepsilon(z')^{I}\right|\leqslant\frac{\delta}{4}
\]
if $I\notin\mathcal{I}$, which also implies that $z^{I}=(z')^{I}.$
This contradicts $z\neq z'.$ If $\mathcal{I}\neq\mathcal{I}',$ there
exists some $I$ such that $I\in\mathcal{I}\cap(\mathcal{I}')^{c}.$
It follows that $|z^{I}-(z')^{I}|\geqslant1/2.$ On the other hand,
we have 
\[
\left|a^{I}-\varepsilon z^{I}\right|\leqslant\frac{\varepsilon-\delta}{2},\ \left|a^{I}-\varepsilon(z')^{I}\right|\leqslant\frac{\delta}{4}.
\]
Therefore, 
\[
|z^{I}-(z')^{I}|\leqslant\frac{1}{2}-\frac{\delta}{4\varepsilon},
\]
which is contradiction. 

(2) First note that $\overline{H_{z}^{\delta;D-1}}$ is the convex
hull of its vertices. Therefore it suffices to show that each vertex
of $\overline{H_{z}^{\delta;D-1}}$ is an element of $\overline{C_{z}^{\delta';D-1}}.$
Let $a$ be a vertex of $\overline{H_{z}^{\delta;D-1}}.$ Without
loss of generality, we may assume that 
\begin{equation}
a^{I}=\begin{cases}
\varepsilon z^{I}+\frac{\varepsilon-\delta}{2}, & \mathrm{if\ }I\in\mathcal{I};\\
\varepsilon z^{I}+\frac{\delta}{4}, & \mathrm{otherwise,}
\end{cases}\label{eq: vertex coordinates}
\end{equation}
where $\mathcal{I}$ is the set of components of $z$ which are integers.
Let $a^{(0)}$ and $a^{(1)}$ be the corresponding vertices of $\overline{H_{z}^{\delta_{1};D-1}}$
and $\overline{H_{z}^{\delta';D-1}}$ respectively (i.e. replacing
$\delta$ by $\delta_{1}$ and $\delta'$ in (\ref{eq: vertex coordinates})
respectively). From direct calculation we know that 
\[
a=\frac{\delta-\delta'}{\delta_{1}-\delta'}a^{(0)}+\frac{\delta_{1}-\delta}{\delta_{1}-\delta'}a^{(1)}.
\]
Therefore $a\in\overline{C_{z}^{\delta';D-1}}.$
\end{proof}

Lemma \ref{lem: basic feature of geometric schemes} enables us to
develop the same construction over the geometric scheme $\mathcal{C}^{\delta;D-1}$
as we have done over $\mathcal{C}^{\delta;D}$. More precisely, we
define $L^{\delta;D-1},$ $\tau_{k}^{\delta;D-1},$ $\mathbf{m}^{\delta;D-1},$
admissible chains, and the quantity $\mathfrak{s}_{D-1}$ for words
in $\mathcal{W}^{D-1}$ over the alphabet $V_{D-1}$ in the same way
as before (here we allow the empty word as the path may not necessarily
visit $\mathcal{C}^{\delta;D-1}$, and we set $\mathfrak{s}_{D-1}(\emptyset)=0$
in this case). Moreover, the same argument as in the proof of Lemma
\ref{lem: proof of stability} shows that $\mathfrak{s}_{D-1}$ is
a stable quantity over $V_{D-1}$ with respect to the geometric schemes
$\mathcal{C}^{\delta;D-1}$ ($0<\delta\leqslant\delta_{1}$). Finally
we define 
\[
\delta_{2}=\frac{1}{2}\sup\left\{ 0<\delta<\delta_{1}:\ \mathfrak{s}_{D-1}(\mathbf{m}^{\delta';D-1})=\mathrm{const.}\ \forall\delta'\leqslant\delta\right\} >0
\]
as in (\ref{eq: choice of delta_1}). 
\begin{rem}
We construct the convex hull of $H_{z}^{\delta;D-1}$ and $H_{z}^{\delta_{1};D-1}$
instead of simply using the cube $H_{z}^{\delta;D-1}$ because $\mathfrak{s}_{D-1}$
will no longer be a stable quantity if we use the latter. Moreover,
it will be  difficult to show stability for other attempts to construct
a stable quantity without property (2) of Lemma \ref{lem: basic feature of geometric schemes}. 
\end{rem}

We carry on the construction recursively. Assume that we have constructed
$\mathcal{C}^{\delta_{j};D-j+1}$ for $j=1,\cdots,i.$ For $\delta\leqslant\delta_{i}$
and $z\in V_{i},$ we construct $H_{z}^{\delta;D-i}$ by (\ref{eq: constructing the cube})
and define $C_{z}^{\delta;D-i}$ to be the convex hull of $H_{z}^{\delta;D-i}$
and $H_{z}^{\delta_{i};D-i}.$ The geometric scheme of $\{\mathbf{1}+C_{z}^{\delta;D-i}:\ z\in V_{D-i}\}$
is denoted by $\mathcal{C}^{\delta;D-i}.$ Lemma \ref{lem: basic feature of geometric schemes}
holds in exactly the same way as before. Moreover, we construct the
stable quantity $\mathfrak{s}_{D-i}$ over $V_{D-i}$ and define $\delta_{i+1}>0$
accordingly from the stability property. 

Therefore, we obtain $D$ geometric schemes $\mathcal{C}^{\delta_{i};D-i+1}$
($i=1,\cdots,D$) from the signature $g,$ where $\delta_{1}>\cdots>\delta_{D}>0.$
Let 
\[
\mathcal{C}^{\varepsilon}(g)=\bigcup_{i=1}^{D}\mathcal{C}^{\delta_{i};D-i+1}
\]
be the totality of domains in each of them. $\mathcal{C}^{\varepsilon}(g)$
will be our single geometric scheme on the $\varepsilon$-scale in
which we are going to construct the polygonal approximation. The notation
emphasizes that it is determined by the signature $g$ as well as
its dependence on $\varepsilon.$ The previous Figure \ref{figure 1} also illustrates
the construction of $\mathcal{C}^{\varepsilon}(g)$ when $D=2.$

\begin{lem}
\label{lem: disjointness of final scheme}Let $A,B\in\mathcal{C}^{\varepsilon}(g),$
then $\overline{A}\cap\overline{B}=\emptyset.$\end{lem}
\begin{proof}
The case when $A,B$ come from the same geometric scheme $\mathcal{C}^{\delta_{i};D-i+1}$
is contained in Lemma \ref{lem: basic feature of geometric schemes}.
Now suppose that $A=C_{z}^{\delta_{i};D-i+1}$ and $B=C_{w}^{\delta_{j};D-j+1}$
for some $i<j$ with $z\in V_{D-i+1}$ and $w\in V_{D-j+1}$ respectively.
By definition there exists some $I$ such that $z^{I}$ is an integer
and $w^{I}$ is a half-integer. If $a\in\overline{A}\cap\overline{B},$
from the construction of our domains we know that 
\[
\left|a^{I}-\varepsilon z^{I}\right|\leqslant\frac{\varepsilon-\delta_{i}}{2},\ \left|a^{I}-\varepsilon w^{I}\right|\leqslant\frac{\delta_{j-1}}{4}.
\]
But since $\delta_{j-1}\leqslant\delta_{i}$, we obtain that 
\[
\left|z_{k}-w_{k}\right|\leqslant\frac{1}{2}-\frac{\delta_{i}}{4\varepsilon},
\]
which is a contradiction. Therefore, $\overline{A}\cap\overline{B}=\emptyset.$
\end{proof}

So far we have not seen how the stable quantities $\mathfrak{s}_{i}$
($i=1,\cdots,D$) play a role in our reconstruction problem because
the previous construction certainly applies to arbitrary stable quantities
with respect to the corresponding geometric schemes. This will be
clear in the next subsection.

\subsection{The Convergence}

Recall from Section 3 that $W_{0}$ is the space of continuous paths
in $E_{\lfloor p\rfloor}$ which do not stay constant over any positive
time period, and modulo reparametrization $(W_{0}^{\sim},d)$ is a
metric space. Moreover, the set of equivalence classes for tree-reduced
weakly geometric $p$-rough paths is canonically embedded in $W_{0}^{\sim}.$

Let $\mathbf{X}$ be a simple tree-reduced weakly geometric $p$-rough
path with signature $g.$ Our goal is to reconstruct $\mathbf{X}$
modulo reparametrization, i.e. the equivalence class $[\mathbf{X}]$.

As before, in the geometric scheme $\mathcal{C}^{\varepsilon}(g),$
we define the word 
\[
\mathbf{m}^{\varepsilon}=(m_{0}^{\varepsilon}=0,m_{1}^{\varepsilon},\cdots,m_{L^{\varepsilon}}^{\varepsilon})
\]
over the alphabet $V=V_{1}\cup\cdots\cup V_{D}$ corresponding to
the ordered sequence of domains in $\mathcal{C}^{\varepsilon}$ visited
by $\mathbf{X}$, which is again determined by the signature explicitly.
We also define the corresponding entry times $\{\tau_{k}^{\varepsilon}:\ 0\leqslant k\leqslant L^{\varepsilon}\},$
which is \textit{not} determined by the signature as it is invariant
under reparametrization. 

Now we construct the associated polygonal approximation of $\mathbf{X}$
by joining the points in $\varepsilon\mathbf{m}^{\varepsilon}$ in
order and parametrizing it according to the successive entry times
$\tau_{k}^{\varepsilon}$. More precisely, we define
\[
\mathbf{X}_{t}^{\varepsilon}=\frac{\tau_{k}^{\varepsilon}-t}{\tau_{k}^{\varepsilon}-\tau_{k-1}^{\varepsilon}}(\mathbf{1}+\varepsilon m_{k-1}^{\varepsilon})+\frac{t-\tau_{k-1}^{\varepsilon}}{\tau_{k}^{\varepsilon}-\tau_{k-1}^{\varepsilon}}(\mathbf{1}+\varepsilon m_{k}^{\varepsilon})
\]
for $t\in[\tau_{k-1}^{\varepsilon},\tau_{k}^{\varepsilon}]$ if $1\leqslant k\leqslant L^{\varepsilon}$
and 
\[
\mathbf{X}_{t}^{\varepsilon}=\mathbf{1}+\varepsilon m_{L^{\varepsilon}}^{\varepsilon}
\]
for $t\in\left[\tau_{L^{\varepsilon}}^{\varepsilon},1\right].$ 

Apparently the path $\mathbf{X}^{\varepsilon}$ is not determined
by the signature $g.$ However, its equivalence class $[\mathbf{X}^{\varepsilon}],$
being regarded as an element in $W_{0}^{\sim},$ is reconstructed by
$g$ since it is determined by the word $\mathbf{m}^{\varepsilon}$
only.

Our reconstruction for the non-self-intersecting case will be finished
by proving that $[\mathbf{X}^{\varepsilon}]$ converges to $[\mathbf{X}]$
in $(W_{0}^{\sim},d)$ as $\varepsilon\rightarrow0$. This is based
on the following crucial fact.
\begin{prop}
\label{prop: impossible to travel through long tunnels}Let $T^{\varepsilon}=\left(\bigcup_{C\in\mathcal{C}^{\varepsilon}}C\right)^{c}$
be the set of tunnels for the geometric scheme $\mathcal{C}^{\varepsilon}.$
Then there do not exist $s<t$ such that 
\begin{equation}
\left|\mathbf{X}_{t}-\mathbf{X}_{s}\right|\geqslant33D^{\frac{3}{2}}\varepsilon\label{eq: assuming traveling through long tunnel}
\end{equation}
 and $\mathrm{Im}(\mathbf{X}|_{[s,t]})\subset T^{\varepsilon}.$\end{prop}
\begin{proof}
Assume on the contrary that such $s,t$ exist. Since $\mathrm{Im}(\mathbf{X}|_{[s,t]})\subset T^{\varepsilon},$
we know that $\mathbf{X}$ does not visit any domain in $\mathcal{C}^{\varepsilon}(g)$
during $[s,t].$ For $1\leqslant i\leqslant D,$ let $C_{i}$ be the
last domain in $\mathcal{C}^{\delta_{i};D-i+1}$ visited by $\mathbf{X}$
before $s$, and let $C_{i}'$ be the first domain in $\mathcal{C}^{\delta_{i};D-i+1}$
visited by $\mathbf{X}$ after $t.$ It might be possible that only
one of them exists or even both do not exist. Let $z_{i},z_{i}'$
be their centers respectively. Since there are at most $2D$ of them,
there exists some $1\leqslant k\leqslant3D,$ such that 
\[
\mathbf{1}+\varepsilon z_{i},\mathbf{1}+\varepsilon z_{i}'\notin A_{11(k-1)\sqrt{D}\varepsilon,11k\sqrt{D}\varepsilon}(\mathbf{X}_{s})
\]
for all $1\leqslant i\leqslant D$, where 
\[
A_{r,R}(a):=\{b\in E_{\lfloor p\rfloor}:\ r<|b-a|<R\}
\]
denotes the open $(r,R)$-annulus around $a.$

By continuity and (\ref{eq: assuming traveling through long tunnel}),
there exist $s_{1}<t_{1}\in[s,t]$ such that 
\[
|\mathbf{X}_{s_{1}}-\mathbf{X}_{s}|=11(k-1)\sqrt{D}\varepsilon
\]
and 
\[
|\mathbf{X}_{t_{1}}-\mathbf{X}_{s}|=11k\sqrt{D}\varepsilon.
\]
Let 
\[
t_{2}=\inf\{u\in[s_{1},t_{1}]:\ |\mathbf{X}_{u}-\mathbf{X}_{s}|\geqslant(11k-5)\sqrt{D}\varepsilon\}
\]
and
\[
s_{2}=\sup\{u\in[s_{1},t_{2}]:\ |\mathbf{X}_{u}-\mathbf{X}_{s}|\leqslant(11k-6)\sqrt{D}\varepsilon\}.
\]
It follows that 
\begin{equation}
|\mathbf{X}_{t_{2}}-\mathbf{X}_{s_{2}}|\geqslant\sqrt{D}\varepsilon\label{eq: the contradiction}
\end{equation}
and $\mathrm{Im}(\mathbf{X}|_{[s_{2},t_{2}]})\subset\overline{A}_{(11k-6)\sqrt{D}\varepsilon,(11k-5)\sqrt{D}\varepsilon}$
(the closed annulus).

Now assume that there exists some $u\in[s_{2},t_{2}]$ with $\mathbf{X}_{u}\in \mathbf{1}+\varepsilon K_{D}$
(recall that $K_{D}$ is the open $D$-skeleton). It follows that
there exists some $\delta<\delta_{1}$ such that $\mathbf{X}_{u}\in\mathbf{1}+H_{z}^{\delta;D}$
for some $z\in V_{D}.$ From the construction of $s_{2},t_{2},$ we
have 
\begin{equation}
|z-z_{1}|,|z-z_{1}'|\geqslant4\sqrt{D}.\label{eq: distance between z and z_1,z_1'}
\end{equation}
Since $\mathrm{Im}(\mathbf{X}|_{[s,t]})\subset T^{\varepsilon},$
we conclude that $\mathbf{m}^{\delta_{1};D}$ must be a proper subword
of $\mathbf{m}^{\delta;D}.$ More precisely, if we write $\mathbf{m}^{\delta_{1};D}$
in the form 
\[
\mathbf{m}^{\delta_{1};D}=(\mathbf{z}_{1},z_{1},z_{1}',\mathbf{z}_{1}'),
\]
where $\mathbf{z}_{1},\mathbf{z}_{1}'$ are the sections of $\mathbf{m}^{\delta_{1};D}$
before $z_{1}$ and after $z_{1}'$ respectively, then the word 
\[
\mathbf{w}=(\mathbf{z}_{1},z_{1},z,z_{1}',\mathbf{z}_{1}')\in\mathcal{W}^{D}
\]
is a subword of $\mathbf{m}^{\delta;D}.$ Therefore, by definition
we have $\mathfrak{s}_{D}(\mathbf{w})\leqslant\mathfrak{s}_{D}(\mathbf{m}^{\delta;D}).$

On the other hand, let $c$ be a maximal admissible chain in $\mathbf{m}^{\delta_{1};D}$,
and let $m_{1},m_{1}'$ be the last and first letters in $c\cap(\mathbf{z}_{1},z_{1})$
and $c\cap(z_{1}',\mathbf{z}_{1}')$ respectively. If $|z_{1}-m_{1}|,|z_{1}'-m_{1}'|\leqslant2\sqrt{D}$,
from (\ref{eq: distance between z and z_1,z_1'}) we know that 
\[
|z-m_{1}|,|z-m_{1}'|\geqslant2\sqrt{D},
\]
which implies that $(c\cap(\mathbf{z}_{1},z_{1}),z,c\cap(z_{1}',\mathbf{z}_{1}'))$
is an admissible chain in $\mathbf{w}$, and thus 
\[
\mathfrak{s}_{D}(\mathbf{w})\geqslant\mathfrak{s}_{D}(\mathbf{m}^{\delta_{1};D})+1.
\]
Similarly, if $|z_{1}-m_{1}|\leqslant2\sqrt{D},$ $|z_{1}'-m_{1}'|\geqslant2\sqrt{D}$
or if $|z_{1}-m_{1}|\geqslant2\sqrt{D}$, $|z_{1}'-m_{1}'|\leqslant2\sqrt{D},$
we have 
\[
\mathfrak{s}_{D}(\mathbf{w})\geqslant\mathfrak{s}_{D}(\mathbf{m}^{\delta_{1};D})+2.
\]
And if $|z_{1}-m_{1}|,|z_{1}'-m_{1}'|\geqslant2\sqrt{D},$ we have
\[
\mathfrak{s}_{D}(\mathbf{w})\geqslant\mathfrak{s}_{D}(\mathbf{m}^{\delta_{1};D})+3.
\]
In other words, we conclude that $\mathfrak{s}_{D}(\mathbf{m}^{\delta_{1};D})<\mathfrak{s}_{D}(\mathbf{w})\leqslant\mathfrak{s}_{D}(\mathbf{m}^{\delta;D}).$
But this contradicts the construction of $\delta_{1}.$ Therefore,
$\mathrm{Im}(\mathbf{X}|_{[s_{2},t_{2}]})\subset \mathbf{1}+\varepsilon C_{D-1}$ (recall
that $C_{D-1}$ is the closed $(D-1)$-skeleton). 

Based on the construction of $\delta_{2},$ the same argument shows
that $\mathrm{Im}(\mathbf{X}|_{[s_{2},t_{2}]})\subset \mathbf{1}+\varepsilon C_{D-2}.$ Recursively,
we conclude that $\mathrm{Im}(\mathbf{X}|_{[s_{2},t_{2}]})\subset \mathbf{1}+\varepsilon C_{0}.$
Note that in each step $i$ of the recursive argument, we should also
consider the cases when one of $z_{i},z_{i}'$ does not exist and
both of them do not exist, but these two cases are apparently simpler
than the general discussion before. 

Finally, since $C_{0}$ is a discrete space, from continuity we have
$\mathbf{X}|_{[s_{2},t_{2}]}=\mathrm{const},$ which contradicts (\ref{eq: the contradiction}). 

Now the proof is complete. 
\end{proof}

Finally, we are in a position to prove the following convergence result.
\begin{thm}
\label{thm: the convergence}For every $\varepsilon>0,$ we have
\begin{equation}
\sup_{t\in[0,1]}\left|\mathbf{X}_{t}^{\varepsilon}-\mathbf{X}_{t}\right|\leqslant68D^{\frac{3}{2}}\varepsilon.\label{eq: the convergence}
\end{equation}
In particular, 
\[
\lim_{\varepsilon\rightarrow0}[\mathbf{X}^{\varepsilon}]=[\mathbf{X}]\ \mathrm{in}\ (W_{0}^{\sim},d).
\]
\end{thm}
\begin{proof}
Consider the time interval $[\tau_{k-1}^{\varepsilon},\tau_{k}^{\varepsilon}].$
Let $C_{k-1}^{\varepsilon},C_{k}^{\varepsilon}$ be the domain in
$\mathcal{C}^{\varepsilon}(g)$ corresponding to the entry times $\tau_{k-1}^{\varepsilon},\tau_{k}^{\varepsilon}$
respectively. Define 
\[
t^{*}=\sup\{t\in[\tau_{k-1}^{\varepsilon},\tau_{k}^{\varepsilon}]:\ \mathbf{X}_{u}\in C_{k-1}^{\varepsilon}\}.
\]
If $|m_{k}^{\varepsilon}-m_{k-1}^{\varepsilon}|>34D^{3/2},$ then
we have 
\[
|\mathbf{X}_{t^{*}}-\mathbf{X}_{\tau_{k}^{\varepsilon}}|\geqslant34D^{\frac{3}{2}}\varepsilon-2\times\frac{\sqrt{D}}{2}\varepsilon\geqslant33D^{\frac{3}{2}}\varepsilon.
\]
Moreover, we know that $\mathrm{Im}(\mathbf{X}(|_{[t^{*},\tau_{k}^{\varepsilon}]})\subset T^{\varepsilon}.$
This contradicts Proposition \ref{prop: impossible to travel through long tunnels}.
Therefore, 
\begin{equation}
|m_{k}^{\varepsilon}-m_{k-1}^{\varepsilon}|\leqslant34D^{3/2}.\label{eq: neighboring visits cannot be too far}
\end{equation}

Let $t\in[\tau_{k-1}^{\varepsilon},\tau_{k}^{\varepsilon}].$ If $|\mathbf{X}_{t}-(\mathbf{1}+\varepsilon m_{k-1}^{\varepsilon})|>68D^{3/2}\varepsilon,$
from (\ref{eq: neighboring visits cannot be too far}) we know that
\[
|\mathbf{X}_{t}-(1+\varepsilon m_{k}^{\varepsilon})|\geqslant34D^{\frac{3}{2}}\varepsilon.
\]
Therefore, 
\[
\overline{B}(\mathbf{X}_{t},33D^{\frac{3}{2}}\varepsilon)\bigcap\left(\overline{C_{k-1}^{\varepsilon}}\bigcup\overline{C_{k}^{\varepsilon}}\right)=\emptyset.
\]
Define 
\[
t'=\inf\{t\in[t,\tau_{k}^{\varepsilon}]:\ \mathbf{X}_{t}\in\partial\overline{B}(\mathbf{X}_{t},33D^{\frac{3}{2}}\varepsilon)\}.
\]
It follows that $|\mathbf{X}_{t'}-\mathbf{X}_{t}|=33D^{3/2}\varepsilon$
and $\mathrm{Im}(\mathbf{X}|_{[t,t']})\subset T^{\varepsilon}.$ This
is again a contradiction to Proposition \ref{prop: impossible to travel through long tunnels}.
Therefore, 
\[
|\mathbf{X}_{t}-(\mathbf{1}+\varepsilon m_{k-1}^{\varepsilon})|\leqslant68D^{\frac{3}{2}}\varepsilon.
\]
The same argument shows that 
\[
|\mathbf{X}_{t}-(\mathbf{1}+\varepsilon m_{k}^{\varepsilon})|\leqslant68D^{\frac{3}{2}}\varepsilon.
\]

Finally, from the construction of $\mathbf{X}^{\varepsilon},$ we
obtain that 
\[
\sup_{t\in[\tau_{k-1}^{\varepsilon},\tau_{k}^{\varepsilon}]}|\mathbf{X}_{t}^{\varepsilon}-\mathbf{X}_{t}|\leqslant68D^{\frac{3}{2}}\varepsilon.
\]
The same argument shows that 
\[
\sup_{t\in[\tau_{L^{\varepsilon}}^{\varepsilon},1]}|\mathbf{X}_{t}^{\varepsilon}-\mathbf{X}_{t}|\leqslant34D^{\frac{3}{2}}\varepsilon.
\]
Therefore, (\ref{eq: the convergence}) holds.
\end{proof}

\begin{rem}
\label{rem: adjusting the approximation}From a technical point one
might observe that $\mathbf{X}^{\varepsilon}$ is actually not an
element of $W_{0}$ since it stays constant on $[\tau_{L^{\varepsilon}}^{\varepsilon},1]$. However, we can modify $\mathbf{X}^{\varepsilon}$ over $[\tau_{L^{\varepsilon}}^{\varepsilon},1]$
to be non-constant in the $\varepsilon$-scale in any arbitrary way,
so that the result of Theorem \ref{thm: the convergence} is still
valid. 
\end{rem}

\section{The Reconstruction: General Case}

Now we turn to the general case. Let $\mathbf{X}$ be a tree-reduced
weakly geometric $p$-rough path with signature $g.$ 

For each $N\geqslant\lfloor p\rfloor,$ the truncated signature path
of $\mathbf{X}$ up to degree $N$ is denoted by $X^{(N)}.$ If $I$
is a word over $\{1,\cdots,d\}$ with $|I|\leqslant N,$  $X_{t}^{I}$
denotes the $I$-th component of $X_{t}^{(N)}.$

The main ingredient here is to construct a degree $N(g)\geqslant\lfloor p\rfloor$
from the signature $g$, so that we can conclude: $|X_{t}^{I}|\leqslant1/2$
for all words $I$ with $|I|>N(g)$ and all $t\in[0,1].$ Starting
from this, the construction in the last section carries through on
the Euclidean space $E_{N(g)}$ without much difficulty. 

Now let $0<\delta<1/4$ be a parameter. Again $\delta$ should be
regarded as a discrete parameter, and only the direction $\delta\rightarrow0$
matters.

For each $N\geqslant\lfloor p\rfloor$, let $D_{N}$ be the dimension
of the Euclidean space $E_{N}.$ Define $A_{0}^{N}$ to be the set
of points $z=(z^{I})_{0\leqslant|I|\leqslant N}\in E_{N}$ such that
at least one of the components of $z$ is $\pm1/2$ and the rest are
of the form $z^{I}=n/2$ where $n\in\mathbb{Z}.$ Let $A^{N}=\{0\}\cup A_{0}^{N}.$

Given $z\in A^{N},$ let
\[
Q_{z}^{\delta;N}=\left\{ a\in E_{N}:\ |a^{I}-z^{I}|<\frac{1}{2}-\delta\ \mathrm{if}\ z^{I}\in\mathbb{Z}\ \mathrm{and}\ |a^{I}-z^{I}|<\frac{\delta}{2}\ \mathrm{otherwise}\right\} .
\]
Let $\mathcal{Q}^{\delta;N}$ be the geometric scheme on $E_{N}$
consisting of all the cubes $\mathbf{1}+Q_{z}^{\delta;N}$ ($z\in A^{N}$).
Apparently the closure of these cubes are mutually disjoint. A crucial
feature of $\mathcal{Q}^{\delta;N}$ for us is the following: 
\begin{equation}
\left\{ a\in E_{N}:\ |a^{I}|=\frac{1}{2}\ \mathrm{for\ some\ }I\right\} \subset\bigcup_{0<\delta<\frac{1}{4}}\bigcup_{z\in A_{0}^{N}}Q_{z}^{\delta;N}.\label{eq: feature of scheme in general case}
\end{equation}
Figure \ref{figure 2} illustrates the construction of $\mathcal{Q}^{\delta;N}$
when $D_{N}=2.$

\begin{figure} 
\begin{center} 
\includegraphics[scale=0.3]{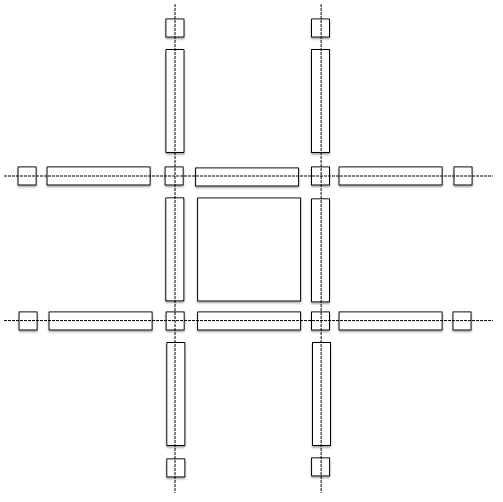}
\protect\caption{This figure illustrates the construction of $\mathcal{Q}^{\delta;N}$ when $D_N=2$.}\label{figure 2} 
\end{center}
\end{figure}

Define $L^{\delta;N}$ and the word 
\[
\mathbf{m}^{\delta;N}=(m_{0}^{\delta;N}=0,m_{1}^{\delta;N},\cdots,m_{L^{\delta;N}}^{\delta;N})
\]
corresponding to the ordered sequence of cubes in $\mathcal{Q}^{\delta;N}$
visited by the truncated signature path $X^{(N)}$ as before. 

As in Section 5.1, for $l\geqslant0$, let $\widetilde{\mathcal{W}}_{l}^{N}$
be the set of words $\mathbf{n}=(n_{0}=0,n_{1},\cdots n_{l})$ over
the alphabet $A^{N}$ such that $n_{k}\neq n_{k-1}$ for $1\leqslant k\leqslant l.$
After pre-specifying a countable family of generating one forms over
the geometric scheme $\mathcal{Q}^{\delta;N}$ on $E_{N},$ we define
$\chi(g;l,\mathbf{n})$ in the same way as (\ref{eq: defining chi}),
and also define 
\[
l^{\delta;N}=\sup\{l\geqslant0:\ \chi(g;l,\mathbf{n})=1\ \mathrm{for\ some\ }\mathbf{n}\in\widetilde{\mathcal{W}}_{l}^{N}\},
\]
where we set $l^{\delta;N}=0$ if such $\mathbf{n}$ does not exist
for every $l$. According to Proposition \ref{prop: existence of one forms}
(2), we know that $l^{\delta;N}<\infty.$ Let $\mathbf{n}^{\delta;N}$
be a word in $\widetilde{\mathcal{W}}_{l^{\delta;N}}^{N}$ such that
$\chi(g;l^{\delta;N},\mathbf{n}^{\delta;N})=1.$ Again $l^{\delta;N}$
and $\mathbf{n}^{\delta;N}$ are determined by the signature $g$
explicitly.

Now the key observation is the following: due to the decay of the homogeneous
signature path of high degree, when $N$ becomes large, the high dimensional
component of every letter in $\mathbf{n}^{\delta;N}$ becomes zero.
To be more precise, given $N<N'$, let 
\[
\pi_{N}^{N'}:\ E_{N'}\rightarrow\bigoplus_{i=N+1}^{N'}(\mathbb{R}^{d})^{\otimes i}
\]
be the projection onto the component of degree higher than $N.$ If
$\mathbf{n}$ is a word with letters in $E_{N'},$ $\pi_{N}^{N'}(\mathbf{n})$
is the word obtained by projecting every letter in $\mathbf{n}$ accordingly.
Then we have the following result.
\begin{lem}
\label{lem: picking out N}There exists some $N_{0}\geqslant\lfloor p\rfloor$
independent of $\delta,$ such that for any $N>N_{0}$ and $0<\delta<1/4$,
we have 
\begin{equation}
\pi_{N_{0}}^{N}(\mathbf{n}^{\delta;N})=(0,\cdots,0).\label{eq: zero projection}
\end{equation}
\end{lem}
\begin{proof}
According to Lyons' extension theorem (Theorem \ref{thm: Lyons' extension}),
there exists some $N_{0}\geqslant\lfloor p\rfloor$ depending on the
path $\mathbf{X},$ such that $|X_{t}^{I}|\leqslant1/4$ for all $I$
with $|I|>N_{0}$ and all $t\in[0,1].$ If there exist $N>N_{0}$
and $0<\delta<1/4$ such that $\pi_{N_{0}}^{N}(\mathbf{n}^{\delta;N})\neq0,$
then there exists a letter $z\in\mathbf{n}^{\delta;N}$ such that
$z^{I}\neq0$ for some $I$ with $N_{0}<|I|\leqslant N.$ From the
construction of $\mathbf{n}^{\delta;N},$ we conclude that the truncated
signature path $X^{(N)}$ has visited the cube $Q_{z}^{\delta;N}$
(for this it is helpful to see \cite{BG2015}, Lemma 4.2). Therefore,
there exists some $t\in[0,1]$ such that 
\begin{eqnarray*}
|X_{t}^{I}| & \geqslant & |z^{I}|-|X_{t}^{I}-z^{I}|\geqslant\begin{cases}
1-\left(\frac{1}{2}-\delta\right)>\frac{1}{4}, & \mathrm{if\ }z^{I}\ \mathrm{is\ an\ integer};\\
\frac{1}{2}-\frac{\delta}{2}>\frac{1}{4}, & \mathrm{\mathrm{if}\ }z^{I}\ \mathrm{is\ a\ }\mbox{half-integer},
\end{cases}
\end{eqnarray*}
which is a contradiction.
\end{proof}

Here a big difference from the non-self-intersecting case is that
$l^{\delta;N}$ can be strictly less than $L^{\delta;N},$ and $\mathbf{n}^{\delta;N}$
can just be some proper subword of $\mathbf{m}^{\delta;N}$. Figure
\ref{figure 3} illustrates this possibility when $D_{N}=2.$ In this example, the reason is that any one form supported on the top long box integrates to zero as the underlying path cancels itself exactly inside this box. However,
from Proposition \ref{prop: existence of one forms} we know that
if $l^{\delta;N}=L^{\delta;N},$ then $\mathbf{n}^{\delta;N}$ must
coincide with $\mathbf{m}^{\delta;N}$. This is an important fact.

\begin{figure} 
\begin{center} 
\includegraphics[scale=0.45]{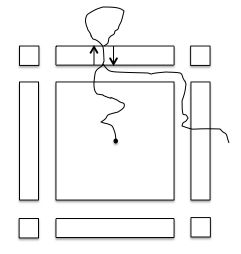}
\protect\caption{This figure illustrates the possibility that $\mathbf{n}^{\delta;N}$ is a proper subword of $\mathbf{m}^{\delta;N}$ when $D_N=2$. Here the underlying path is apparently tree-reduced. However, $\mathbf{n}^{\delta;N}=((0,0),(1/2,0))$ and $\mathbf{m}^{\delta;N}=((0,0),(0,1/2),(1/2,0))$.}\label{figure 3} 
\end{center}
\end{figure}

Let $N_{0}$ be given by Lemma \ref{lem: picking out N}. Given $N<N',$
let $l_{N}^{N'}$ be the lifting map 
\begin{eqnarray*}
l_{N}^{N'}:\ E_{N} & \rightarrow & E_{N'}=E_{N}\bigoplus\bigoplus_{i=N+1}^{N'}(\mathbb{R}^{d})^{\otimes i},\\
z & \mapsto & (z,0).
\end{eqnarray*}
If $\mathbf{n}$ is a word with letters in $E_{N},$ $l_{N}^{N'}(\mathbf{n})$
is the word obtained by lifting every letter in $\mathbf{n}$ accordingly.
The following result is important for us. It is where the tree-reduced
property comes in.
\begin{lem}
\label{lem: n=00003Dm when N large}For $N'>N\geqslant N_{0},$ we
have 
\[
\mathbf{m}^{\delta;N'}=l_{N}^{N'}(\mathbf{m}^{\delta;N})
\]
for all $0<\delta<1/4$. In other words, knowing the discrete route
of $X^{(N')}$ in the geometric scheme $\mathcal{Q}^{\delta;N'}$
is equivalent to knowing the one of $X^{(N)}$ in $\mathcal{Q}^{\delta;N}$.
Moreover, for each given $\delta$, when $N$ is large enough we have
$\mathbf{n}^{\delta;N}=\mathbf{m}^{\delta;N}.$\end{lem}
\begin{proof}
Recall from the proof of Lemma \ref{lem: picking out N} that $|X_{t}^{I}|\leqslant1/4$
for $|I|>N_{0}$ and $t\in[0,1].$ Therefore, if $X_{t}^{(N')}\in Q_{z}^{\delta;N'}$
for some $z\in A^{N'},$ then for any $I$ with $N<|I|\leqslant N'$,
\[
|z^{I}|\leqslant|X_{t}^{I}|+|X_{t}^{I}-z^{I}|\leqslant\begin{cases}
\frac{1}{4}+\frac{1}{2}-\delta, & \mathrm{if}\ z^{I}\ \mathrm{is\ an\ integer};\\
\frac{1}{4}+\frac{\delta}{2}, & \mathrm{if\ }z^{I}\ \mathrm{is\ a\ }\mbox{half-integer}.
\end{cases}
\]
It follows that $z^{I}=0.$ Moreover, given $z\in A^{N},$ we have
$X_{t}^{(N)}\in Q_{z}^{\delta;N}$ if and only if $X_{t}^{(N')}\in Q_{l_{N}^{N'}(z)}^{\delta;N'}$
since 
\[
|X_{t}^{I}|=|X_{t}^{I}-0|\leqslant\frac{1}{4}<\frac{1}{2}-\delta
\]
for every $N<|I|\leqslant N'.$

On the other hand, for given $\delta,$ from the previous discussion
we know that $\mathbf{m}^{\delta;N}=l_{N_{0}}^{N}(\mathbf{m}^{\delta;N_{0}})$
for any $N>N_{0}.$ Define $\{\tau_{k}^{\delta;N}:\ 0\leqslant k\leqslant L^{\delta;N}\}$
to be the corresponding entry times. It is then easy to see that $\tau_{k}^{\delta;N}=\tau_{k}^{\delta;N_{0}}$
for every $0\leqslant k\leqslant L^{\delta;N}=L^{\delta;N_{0}}.$
For each $k$ we choose $\tau_{k}^{\delta;N_{0}}<s_{k}<s_{k}'<t_{k}'<t_{k}<\tau_{k+1}^{\delta;N_{0}}$
such that 
\[
\mathrm{Im}\left(X^{(N_{0})}|_{[s_{k},t_{k}]}\right)\subset Q_{m_{k}^{\delta;N_{0}}}^{\delta;N_{0}},
\]
where $\tau_{L^{\delta;N_{0}}+1}^{\delta;N_{0}}:=1.$ Note that for
$N>N_{0}$ we also have 
\[
\mathrm{Im}\left(X^{(N)}|_{[s_{k},t_{k}]}\right)\subset Q_{l_{N_{0}}^{N}(m_{k}^{\delta;N_{0}})}^{\delta;N}.
\]
Now let 
\[
\eta=\inf_{0\leqslant k\leqslant L^{\delta;N_{0}}}\{s_{k}'-s_{k},t_{k}'-s_{k}',t_{k}-t_{k}'\}.
\]
Since $\mathbf{X}$ is a tree-reduced weakly geometric $p$-rough
path, by a compactness argument (c.f. \cite{BGLY2014}, Lemma 19) we
know that there exists some $N_{1}>N_{0},$ such that $X_{s}^{(N)}\neq X_{t}^{(N)}$
for any $(s,t)\in[0,1]$ with $|t-s|\geqslant\eta$ and any $N>N_{1}$.
Keeping this separation property in mind, by applying exactly the
same argument as in the proof of Proposition \ref{prop: existence of one forms},
we know that for each $N>N_{1},$ the extended signature for $X^{(N)}$
along certain compactly supported $C^{\alpha}$-one forms over the
word $\mathbf{m}^{\delta;N}$ is nonzero. Therefore, together with
Proposition \ref{prop: existence of one forms} (2), we conclude that
$l^{\delta;N}=L^{\delta;N}$ and $\mathbf{n}^{\delta;N}=\mathbf{m}^{\delta;N}.$
\end{proof}

For each $\delta,$ we define 
\[
N(g;\delta)=\inf\{N\geqslant\lfloor p\rfloor:\ \pi_{N}^{N'}(\mathbf{n}^{\delta;N'})=(0,\cdots,0)\ \forall N'>N\},
\]
and define 
\begin{equation}\label{eqn: construction of N(g)}
N(g)=\sup_{0<\delta<\frac{1}{4}}N(g;\delta).
\end{equation}
From Lemma \ref{lem: picking out N} we have $N(g)\leqslant N_{0}$.
Note that $N(g)$ is reconstructed from the signature $g$.
Such $N(g)$ will serve our purpose. Namely, we have the following
result.
\begin{prop}
\label{prop: decay}For any $I$ with $|I|>N(g)$ and $t\in[0,1],$
we have $|X_{t}^{I}|\leqslant1/2$.\end{prop}
\begin{proof}
Assume on the contrary that there exists some $I$ with $|I|>N(g)$
and some $t$ such that $|X_{t}^{I}|>1/2.$ Let $N=|I|\lor N_{0}.$
By continuity, there exists some $t'\in(0,t)$ such that $|X_{t'}^{I}|=1/2.$
Without loss of generality we assume that $X_{t'}^{I}=1/2.$ According
to (\ref{eq: feature of scheme in general case}), there exist some
$0<\delta<1/4,$ $z\in A_{0}^{N},$ such that $X_{t'}^{(N)}\in Q_{z}^{\delta;N}.$
Therefore, $z$ is a letter in $\mathbf{m}^{\delta;N}$. Moreover,
we have 
\[
\left|\frac{1}{2}-z^{I}\right|\leqslant\begin{cases}
\frac{1}{2}-\delta, & \mathrm{if}\ z^{I}\ \mathrm{is\ an\ integer};\\
\frac{\delta}{2}, & \mathrm{if}\ z^{I}\ \mathrm{is\ a\ }\mbox{half-integer},
\end{cases}
\]
which implies that $z^{I}$ must be $1/2.$ 

According to Lemma \ref{lem: n=00003Dm when N large}, we know that
$\mathbf{m}^{\delta;N'}=l_{N}^{N'}(\mathbf{m}^{\delta;N})$ and $\mathbf{n}^{\delta;N'}=\mathbf{m}^{\delta;N'}$
when $N'$ is large enough. Therefore, 
\[
\pi_{N(g)}^{N'}(\mathbf{n}^{\delta;N'})=\pi_{N(g)}^{N'}(\mathbf{m}^{\delta;N'}).
\]
But from the previous discussion we know that the right hand side of the above
identity contains a component $z^{I}=1/2,$ thus $\pi_{N(g)}^{N'}(\mathbf{n}^{\delta;N'})$
is nonzero. This contradicts the definition of $N(g).$
\end{proof}

The final step is to develop exactly the same construction as in the
non-self-intersecting case over the Euclidean space $E_{N(g)}.$ 

The only additional point is to see how to reconstruct the discrete
route of $X^{(N(g))}$ in a given geometric scheme in $E_{N(g)}$
from the signature $g$. Let $\mathcal{C}=\{C_{i}:\ i\in\mathcal{I}\}$
be a given geometric scheme and let $\mathbf{m}^{\mathcal{C}}$ be
the word over the alphabet $\mathcal{I}$ corresponding to the ordered
sequence of domains in $\mathcal{C}$ visited by $X^{(N(g))}$. As
before, in general we cannot expect that $\mathbf{m}^{\mathcal{C}}$
can be reconstructed from $g$ by computing extended signatures in
$E_{N(g)}$ only, and we need to lift the construction to higher degrees.
For each $N>N(g),$ let $\mathcal{C}^{N}$ be the geometric scheme
in $E_{N}$ defined by
\[
\mathcal{C}^{N}=\{C_{i}\times U_{N}:\ i\in\mathcal{I}\},
\]
where $U_{N}$ is the cube in $\oplus_{i=N(g)+1}^{N}(\mathbb{R}^{d})^{\otimes i}$
given by 
\[
U_{N}=\{(z^{I})_{N(g)<|I|\leqslant N}:\ |z^{I}|<1\ \mathrm{for\ every\ }I\}.
\]
From Proposition \ref{prop: decay}, we know that visiting $C_{i}$
by $X^{(N(g))}$ is equivalent to visiting $C_{i}\times U_{N}$ by
$X^{(N)}$, i.e. $\mathbf{m}^{\mathcal{C}}=\mathbf{m}^{\mathcal{C};N}$.
Moreover, the same argument as in the proof of Lemma \ref{lem: n=00003Dm when N large}
shows that $\mathbf{n}^{\mathcal{C};N}=\mathbf{m}^{\mathcal{C};N}$
when $N$ is large enough, where $\mathbf{n}^{\mathcal{C};N}$ is
defined in the same way as $\mathbf{n}^{\delta;N}$ before by computing
extended signatures in $E_{N}.$ Therefore, the word $\mathbf{n}^{\mathcal{C};N}$
is stable as $N\rightarrow\infty$ (here the alphabet is always $\mathcal{I}$
for every $N$), and we obtain that 
\begin{eqnarray*}
\mathbf{m}^{\mathcal{C}} & = & \lim_{N\rightarrow\infty}\mathbf{n}^{\mathcal{C};N}=\mathbf{n}^{\mathcal{C};N_{1}(g)},
\end{eqnarray*}
where 
\[
N_{1}(g)=\inf\{N>N(g):\ \mathbf{n}^{\mathcal{C};N'}=\mathbf{n}^{\mathcal{C};N}\ \forall N'>N\}<\infty.
\]
In particular, $\mathbf{m}^{\mathcal{C}}$ is determined by the signature
$g$ explicitly. 

Finally, if we look back into the discussion in Section 5.2 and 5.3,
once $\mathbf{m}^{\mathcal{C}}$ is known for any geometric scheme
$\mathcal{C}$ we are concerned with, the whole argument relies only
on continuity (not even the rough path nature of the underlying path).
In particular, by applying the same construction of geometric schemes
as in the non-self-intersecting case, we obtain an approximating sequence
$X^{\varepsilon}$ of polygonal paths for the truncated signature
path $X^{(N(g))}$ on $E_{N(g)}$. Let $\mathbf{X}^{\varepsilon}$
be the projection of $X^{\varepsilon}$ onto $E_{\lfloor p\rfloor}$
and let $[\mathbf{X}^{\varepsilon}]$ be the corresponding equivalence
class in $(W_{0}^{\sim},d).$ Then we have the following convergence
result. Remark \ref{rem: adjusting the approximation} applies in
the same way here.
\begin{thm}\label{thm: main reconstruction theorem}
For every $\varepsilon>0,$ we have
\[
\sup_{t\in[0,1]}\left|X_{t}^{\varepsilon}-X_{t}^{(N(g))}\right|_{E_{N(g)}}\leqslant68D_{N(g)}^{\frac{3}{2}}\varepsilon.
\]
In particular, 
\[
\lim_{\varepsilon\rightarrow0}[\mathbf{X}^{\varepsilon}]=[\mathbf{X}]\ \mathrm{in}\ (W_{0}^{\sim},d).
\]

\end{thm}

It is worthwhile to point out that a direct consequence of Theorem \ref{thm: main reconstruction theorem} is the following uniqueness result for the signature of a rough path. 

\begin{cor}\label{cor: uniqueness of simple signature paths}
Let $\mathbf{X},\mathbf{Y}$ be two tree-reduced weakly geometric $p$-rough paths in the sense that their signature paths $\mathbb{X},\mathbb{Y}$ are both simple. Then $\mathbb{X}$ and $\mathbb{Y}$ are equal up to reparametrization if and only if $S(\mathbf{X})_{0,1}=S(\mathbf{Y})_{0,1}$. In particular, in this case $\mathbf{X}$ and $\mathbf{Y}$ are equal up to reparametrization.\end{cor}
\begin{proof}
Suppose that $S(\mathbf{X})_{0,1}=S(\mathbf{Y})_{0,1}=g$. It suffices to show that $\mathbb{X}$ and $\mathbb{Y}$ have the same image.

From the proof of Theorem \ref{thm: main reconstruction theorem}, we know that for every $N>N(g)$ where $N(g)$ is defined by (\ref{eqn: construction of N(g)}), as a trajectory on $E_N$, $X^{(N)}$ (modulo reparametrization) can be reconstructed from $g$, and similarly for $Y^{(N)}$. Therefore, $X^{(N)}$ and $Y^{(N)}$ are equal up to reparametrization.

Now fix $t\in [0,1]$. It follows that for every $N>N(g)$, there exists $s_N\in[0,1]$ such that $X^{(N)}_t=Y^{(N)}_{s_N}$. From compactness we may assume without loss of generality that $s_N\rightarrow s\in[0,1]$ as $N\rightarrow\infty$. By projection we conclude that $$X^{(N)}_t=Y^{(N)}_{s_{N'}}$$ for every $N'>N>N(g)$. By sending $N'\rightarrow\infty$, we obtain that $X^{(N)}_t=Y^{(N)}_s$ for every given $N>N(g)$. This implies that $$\mathbb{X}_t=\mathbb{Y}_s\in\mathrm{Im}(\mathbb{Y}).$$
\end{proof}

\begin{rem}
Corollary \ref{cor: uniqueness of simple signature paths} does not cover the general uniqueness result in \cite{BGLY2014} (c.f. Theorem \ref{thm: uniqueness result}) as the signature is not able to detect any tree-like pieces of the underlying rough path. However, Corollary \ref{cor: uniqueness of simple signature paths} is indeed the key ingredient in proving the general uniqueness result (the necessity part: trivial signature implies being tree-like) as it immediately leads to a canonical real tree structure on the signature group $\mathfrak{S}_p$. The underlying rough path is then realized as a continuous loop in the real tree $\mathfrak{S}_p$. 
\end{rem}

\section{Final Remarks}

We give a few remarks in the following to conclude the present
paper.

1. In summary, the general idea of our reconstruction consists of
two parts: the signature determines the discrete route in \textit{any
}given geometric scheme, and the discrete route gives back our underlying
path by refining the geometric scheme in a way \textit{determined
by the signature. }The first part is an easy consequence of the algebraic
structure of signature, while the second part is the key ingredient of
the present work. If we look back into the discussion, it is
not hard to see that the development of the second part relies only
on the continuity of our underlying path. In other words, if the discrete
route of a continuous path is known for any given geometric scheme of the type in Section 5.1, our reconstruction gives back the
original path. This works for all continuous paths without any regularity
assumption.

The reason we cannot expect that the discrete route would yield
the underlying path when the geometric scheme is fine enough and the second part is necessary is that there are positive gaps among the domains
in our geometric scheme. Such construction is essential in our discussion. Removing the gaps ruins the whole argument
as the discrete route is no longer well-defined. On the other hand, one might ask
if we could use finitely many different geometric schemes to cover all the gaps
so that our argument could be simplified. But it seems difficult to
achieve because there is not a canonical way to embed two discrete routes
arising from two different geometric schemes into a single word. Nevertheless,
we expect that there might be some way to do this and our construction
is certainly not the only possibility.

2. We expect that a recursive construction of geometric schemes in
Section 5 is not necessary, and there might be a way to construct
a single geometric scheme $\mathcal{C}^{\delta}$ in one go with
one parameter $\delta$ only (recall that $\varepsilon$ is fixed).
This requires a more difficult construction of a stable quantity as
we cannot expect that every domain in $\mathcal{C}^{\delta}$ is expanding
as $\delta\rightarrow0$ (it is hard to see whether the quantity we
have constructed in Section 5.2 is stable in this case).

3. The reconstruction problem is essentially (countably) infinite
dimensional. As we are looking for a universal way to treat every
path equally, we should expect that all information of signature as
an infinite dimensional object is to be used. From a practical side,
one might wonder if in each $\varepsilon$-step, we could truncate
our construction to a situation involving only finite information
(depending on $\varepsilon$) and convergence still hold. We expect
that this is possible although we do not pursue along this direction
in our work and we are more focused on the theoretical side. 

However, we want to emphasize that even with the possibility of proving
convergence, we expect that any attempt to obtain a quantitative error
estimate for a \textit{finite} scheme reconstruction should involve
an \textit{a priori control} on the behavior of the class of paths
we are considering (for instance, an a prior uniform regularity estimate
or an a priori  control on certain geometric property). More precisely, it is unlikely
to expect the existence of a quantitative result of the following
type:
\begin{equation}
d(\xi^{\varepsilon}(g),[\mathbf{X}])\leqslant C(\varepsilon,g),\ \forall\varepsilon>0,g\in\mathfrak{S}_{p},\label{eq: non-existence of universal error estimates}
\end{equation}
where $\xi^{\varepsilon}(g)\in W_{0}^{\sim}$ is constructed by using
finitely many components of $g$ through finitely many steps of computation,
and $C(\varepsilon,g)$ is an explicit function depending only on $\varepsilon$
and $g$ such that 
\[
\lim_{\varepsilon\rightarrow0}C(\varepsilon,g)=0,\ \forall g\in\mathfrak{S}_{p}.
\]
This is the nature of the reconstruction problem. To expect an error
estimate like (\ref{eq: non-existence of universal error estimates})
for a finite scheme reconstruction, we believe that the restriction
to an a priori subset of $\mathfrak{S}_{p}$ is necessary. This leaves an
interesting question from the computational point of view, for instance
within our reconstruction method. But this is beyond the scope of the
present paper.

\section*{Acknowledgement} 
The author wishes to thank Professor Terry Lyons for his valuable suggestions on the present work. The author is supported  by the ERC grant (Grant Agreement No.291244 Esig).


\begin{thebibliography}{10}

\bibitem{Aida2011}S. Aida, Vanishing of one-dimensional $L^{2}$-cohomologies
of loop groups, \textit{J. Funct. Anal.} 261 (8): 2164--2213, 2011.

\bibitem{BBL2002}T. Bagby, L. Bos and N. Levenberg, Multivariate
simultaneous approximation, \textit{Constr. Approx.} 18 (3): 569--577,
2002.

\bibitem{BG2015}H. Boedihardjo and X. Geng, The uniqueness of signature
problem in the non-Markov setting, to appear in \textit{Stochastic
Process. Appl.}, 2015.

\bibitem{BGLY2014}H. Boedihardjo, X. Geng, T. Lyons and D. Yang,
The Signature of a rough path: uniqueness, \textit{Adv. Math.} 293: 720--737, 2016.

\bibitem{CDLL2015}T. Cass, B.K. Driver, N. Lim and C. Litterer, On
the integration of weakly geometric rough paths, to appear in \textit{J.
Math. Soc. Japan}, 2015.

\bibitem{CF2010}T. Cass and P. Friz, Densities for rough differential
equations under Hörmander\textquoteright s condition, \textit{Ann.
of Math}. 171 (3): 2115--2141, 2010.

\bibitem{Chen1954}K.T. Chen, Iterated integrals and exponential homomorphisms,
\textit{Proc. London Math. Soc.} 4 (3): 502--512, 1954.

\bibitem{Chen1957}K.T. Chen, Integration of paths, geometric invariants
and a generalized Baker-Hausdorff formula, \textit{Ann. of Math}.
65 (1): 163--178, 1957.

\bibitem{Chen1958}K.T. Chen, Integration of paths-a faithful representation
of paths by non-commutative formal power series, \textit{Trans. Amer.
Math. Soc. }89: 395--407, 1958. 

\bibitem{FV2010}P. Friz and N. Victoir, \textit{Multidimensional
stochastic processes as rough paths}, Cambridge Studies of Advanced
Mathematics, Vol. 120, Cambridge University Press, 2010.

\bibitem{GQ2015}X. Geng and Z. Qian, On an inversion theorem for
Stratonovich\textquoteright s signatures of multidimensional diffusion
paths, to appear in \textit{Ann. Inst. Poincaré} \textit{(B), }2015.

\bibitem{HL2010}B.M. Hambly and T. Lyons, Uniqueness for the signature
of a path of bounded variation and the reduced path group, \textit{Ann.
of Math}. 171 (1): 109--167, 2010.

\bibitem{LQZ2002}M. Ledoux, Z. Qian and T. Zhang, Large deviations
and support theorem for diffusion processes via rough paths, \textit{Stochastic
Process. Appl. }102 (2): 265--283, 2002. 

\bibitem{LLN2013}D. Levin, T. Lyons and H. Ni, Learning from the
past, predicting the statistics for the future, learning an evolving
system, \textit{arXiv:} 1309.0260, 2013.

\bibitem{LQ2013}Y. Le Jan and Z. Qian, Stratonovich\textquoteright s
signatures of Brownian motion determine Brownian sample paths, \textit{Probab.
Theory Relat. Fields} 157: 209--223, 2013.

\bibitem{Lyons1998}T. Lyons, Differential equations driven by rough
signals, \textit{Rev. Mat. Iberoamericana} 14 (2): 215--310, 1998.

\bibitem{LNO2014}T. Lyons, H. Ni and H. Oberhauser, A feature set
for streams and an application to high-frequency financial tick data,
Proceedings of the 2014 International Conference on Big Data Science
and Computing, ACM, 2014.

\bibitem{LX2015A}T. Lyons and W. Xu, Hyperbolic development and inversion
of signature, \textit{arXiv}: 1507.00286, 2015.

\bibitem{LX2015B}T. Lyons and W. Xu, Inverting the signature of a
path, \textit{arXiv}: 1406.7833, 2015.

\bibitem{Hairer2014}M. Hairer, A theory of regularity structures,
\textit{Invent. Math. }198 (2): 269--504, 2014.

\bibitem{Reutenauer1993}C. Reutenauer, \textit{Free Lie algebras},
Clarendon Press, 1993.\end{thebibliography}
\end{document}